\documentclass[10pt,a4paper]{article}

\usepackage{amsmath,amsfonts} 
\usepackage{amsthm}
\usepackage{amssymb}
\usepackage[makeroom]{cancel}
\usepackage[normalem]{ulem}

\usepackage[colorlinks=true, pdfstartview=FitV, linkcolor=blue,
            citecolor=blue, urlcolor=blue]{hyperref}
\usepackage[usenames]{color}
\usepackage{url}
\makeatletter\def\url@leostyle{%
 \@ifundefined{selectfont}{\def\UrlFont{\sf}}{\def\UrlFont{\scriptsize\ttfamily}}} \makeatother

\usepackage{graphicx}

\DeclareMathOperator*{\argmin}{arg\,min}

\usepackage{algorithm}
\usepackage{algpseudocode}

\newcommand{\bP}{\mathbb{P}}
\newcommand{\bN}{\mathbb{N}}
\newcommand{\bR}{\mathbb{R}}
\newcommand{\cB}{\mathcal{B}}
\newcommand{\cN}{\mathcal{N}}
\newcommand{\cR}{\mathcal{R}}

\newtheorem{theorem}{\bf Theorem}
\newtheorem{lemma}{\bf Lemma}

\newtheorem{proposition}{\bf Proposition}

\newtheorem{definition}{\bf Definition}

\newtheorem{remark}{\bf Remark}

\renewcommand{\d}{\operatorname{d}\!}

\title{The 20-60-20 Rule}

\author{Piotr Jaworski\footnote{Institute of Mathematics, University of Warsaw, Warsaw, Poland.} , Marcin Pitera\footnote{Institute of Mathematics, Jagiellonian University, Cracow, Poland.} , }

\binoppenalty=\maxdimen 
\relpenalty=\maxdimen

\begin{document}
\maketitle
\begin{abstract}
In this paper we discuss an empirical phenomena known as the 20-60-20 rule. It says that if we split the population into three groups, according to some arbitrary benchmark criterion, then this particular ratio implies some sort of balance. From practical point of view, this feature often leads to efficient management or control. We provide a mathematical illustration, justifying the occurrence of this rule in many real world situations. We show that for any population, which could be described using multivariate normal vector, this fixed ratio leads to a global equilibrium state, when dispersion and linear dependance measurement is considered.

\smallskip
{\noindent \small
{\it \bf Keywords:}
20-60-20 Rule, 60/20/20 Principle, 20:60:20, Pareto principle, law of the vital few, the principle of factor sparsity, truncated normal distribution, conditional elliptic distribution.
 \\[.4em]
 {\it \bf MSC2010:} 
60K30, 
91B10, 
91B14, 
91B52.  
60A86, 
62A86, 
}
\end{abstract}

\setcounter{tocdepth}{2} 

\binoppenalty=\maxdimen 
\relpenalty=\maxdimen

\section*{Introduction}
The 20-60-20 rule is an empirical statement. It says that if we want to split the population into three groups, using some arbitrary benchmark criterion, then the ratio of 20\%, 60\% and 20\% proves to give an efficient partition. The division is usually made according to the performance of each element in the population and the groups are referred to as negative, neutral and positive, respectively. The first group relates to elements of population which positively contribute to the considered subject (e.g. effective workers, top sale managers, productive members), while the last one denotes the opposite. The middle set corresponds to the middle part of the population, having average performance. Putting it another way we cluster the population basing on a notion of effectiveness.

The importance of this rule comes from the fact that this particular partition seems to be the most effective one, for many empirical problems. Let us present in details two common illustrations of this phenomena and then comment on the efficiency to make this idea more transparent.

The first example considers sales departments. In almost any big company, the employees of the sales department could be split into three groups, maintaining 20-60-20 ratio. The first group are top performers, who make big profits, even without supervision. The middle group are people who need to be managed to make average but stable profits. The last group are people who are heading towards termination or resignation. They produce no good income, even when supervised.

The second example relates to change capability. If you are willing to make substantial changes in any big institution, then on average 20\% of the people are ready, willing and able to change, while 20\% of people would not accept the change, whatever the cost. The middle 60\% will wait to see how the situation turns out.

Corporations use the 20-60-20 rule widely in management and sales departments \cite{Tyn1999,Rob2009}. One of the practical aspects of this phenomena relates to the fact that different procedures and methods are created to handle the efficiency in positive, negative and neutral group and the 20-60-20 ratio proves to be the most efficient partition. For example, in many problems related to human resource management, one should identify and focus his attention on the middle 60\%, as this group could and should be managed efficiently.

Of course there are countless illustrations of this phenomena. One could consider financial market overall condition, fraud and theft capability among group of people, the structure of electorate, sport performance among athletes, potential of students, patient handling, medical treatments, etc. Please see e.g.~\cite{SlaHol2007, Bol2004, Ann2001, Dup2002,GreWod2005, Hin2002,KamMakKob2005, Bach2005,Cre1993}, where the 20-60-20 ratio is used and the detailed procedures are proposed to handle many practical problems.

The natural question is why this specific 20/60/20 ratio is valid in so many situations? Why not 10/80/10 or 30/40/30? Is this a coincidence, or does it follow from some underlying and fundamental structure of the population?

While very popular among practitioners, no scientific evidence of the 20-60-20 principle has been presented yet, due to the authors knowledge. Consequently, this noteworthy rule become more of a slogan, than the scientific fact.

The possible mathematical illustration of this phenomena, based on the dispersion and linear dependance measurement will be the main topic of this paper. We will show that if a (multivariate) random vector is distributed normally and we do conditioning based on the (quantile function of) first coordinate, then the ratio close to 20/60/20 imply a global equilibrium state, when dispersion and linear dependance measurement is considered. In particular, we prove that this particular partition implies the equality of covariance matrices, for all conditional vectors, implying some sort of global balance in the population. We will also discuss the case of monotone dependance using conditional Kendall $\tau$ and Spearman $\rho$ matrices.

The material is organized as follows. The introduction is followed by a short preliminaries, where we establish basic notations used throughout this paper. Next, in Section~\ref{S:balance} we introduce a mathematical model for the 20-60-20 rule and define {\it the equilibrium state}, using conditional covariance matrices. The 20-60-20 rule for multivariate normal vectors is discussed in Section~\ref{S:206020}. Theorem~\ref{th:1} might be considered as the main result of this paper. Section~\ref{S:different.norm} is devoted to the study of different equilibrium states, obtained using correlation matrices, Kendall $\tau$ matrices and Spearman $\rho$ matrices. In particular we present here some theoretical results, when Spearman $\rho$ matrices are considered and a numerical example, illustrating the 20-60-20 rule for sample data. In Section~\ref{S:non-normal} we discuss shortly what happens if we loose the assumption about normality. The general elliptic case is considered here.

\section{Preliminaries}
Let $(\Omega,\Sigma,\bP)$ be a probability space and let $n\in\bN$. Let us fix an $n$-dimensional continuous random vector $X=(X_{1},\ldots,X_{n})$. We will use
\[
H(x_{1},\ldots,x_{n}):=\bP[X_{1}\leq x_{1},\ldots X_{n}\leq x_{n}],
\]
to denote the corresponding joint distribution function and
\[
H_{i}(x)=\bP[X_{i}\leq x],\quad i=1,2,\ldots, n,
\]
to denote the marginal distribution functions. Given a Borel set $\cB$ in $\bar{\bR}^{n}$ such that
\[
\bP[\{ \omega\in\Omega: (X_{1}(\omega),\ldots,X_{n}(\omega))\in\cB\}]>0
\]
we can define the conditional distribution $H_{\cB}$ for all $(x_{1},\ldots,x_{n})\in\cB$ by
\begin{equation}\label{eq:Hb}
H_{\mathcal{B}}(x_{1},\ldots,x_{n})=\bP[X_{1}\leq x_{1},\ldots,X_{n}\leq x_{n}\mid X\in\mathcal{B}].
\end{equation}
Putting it in another words we truncate the random vector $X$ to the Borel set $\mathcal{B}$.
If necessary, we assume the existence of regular conditional probabilities. In this paper we will assume that $\cB$ is a non-degenerate rectangle, i.e. $\cB\in\cR$, where
\[
\cR:=\{A\in\bar{\bR}^{n}: A=[a_{1},b_{1}]\times [a_{2},b_{2}]\times \ldots \times [a_{n},b_{n}],\ \textrm{where}\ a_{n},b_{n}\in\bar{\bR}\ \textrm{and}\ a_{n}<b_{n}\}.
\]
As we will be mainly interested in quantile-based conditioning on the first coordinate, for $q_1,q_2\in [0,1]$ such that $q_1<q_2$, we shall use notation
\begin{equation}\label{eq:Hpq}
H_{[q_1,q_2]}(x_{1},\ldots,x_{n}):=H_{\cB(q_1,q_2)}(x_{1},\ldots,x_{n}),
\end{equation}
where the conditioning set is given by
\[
\cB(q_1,q_2):=[H_{1}^{-1}(q_1),H_{1}^{-1}(q_2)]\times\bar{\bR}\times \ldots\times \bar{\bR}.
\]
We shall also refer to $H_{[q_1,q_2]}$ as the {\it truncated distribution}, while $\cB(q_1,q_2)$ will be called {\it truncation interval} (see~\cite{JohKotBal1994}). 

Moreover, we will denote by $\mu=(\mu_{1},\ldots,\mu_{n})$ and $\Sigma=\{\sigma^{2}_{ij}\}_{i,j=1,\ldots,n}$, the mean vector and covariance matrix of $X$. Similarly as in formula \eqref{eq:Hb}, given $\cB$, we will use $\mu_{\cB}$ and $\Sigma_{\cB}$ to denote the conditional mean vector and the conditional covariance matrix, i.e. mean vector and conditional covariance matrix of a random vector with distribution $H_{\cB}$. Consequently, as in \eqref{eq:Hpq}, we shall write
\[
\mu_{[q_1,q_2]}:=\mu_{\cB(q_1,q_2)} \quad\textrm{and}\quad \Sigma_{[q_1,q_2]}:=\Sigma_{\cB(q_1,q_2)}.
\]
We will also use $\Phi$ and $\phi$ to denote the distribution and density function of a standard univariate normal distribution, respectively.

\section{The global balance}\label{S:balance}
To split the whole population into three separate groups basing on a notion of effectiveness, we need to make an assumption about the probability distribution of the whole population and the given benchmark, which measures the effectiveness of each element in the population. We will assume that $X\sim \cN(\mu,\Sigma)$, i.e. the population could be described using $n$-dimensional random vector $X=(X_1,\ldots,X_n)$, which is normally distributed with mean vector $\mu$ and covariance matrix $\Sigma$. Furthermore, we will assume that the benchmark level is determined by the first coordinate, i.e. $X_1$. Please note that for multivariate normal this may be a linear combination of all other coordinates. One could look at other coordinates as various factors, which could influence the main benchmark. Note that, if we talk about people measures or abilities, then Gaussian functions, often described as bell curves, are a natural choice.

We will seek for two real numbers $q_1,q_2\in [0,1]$ and the corresponding partition
\[
\cB(0,q_1),\quad \cB(q_1,1-q_2),\quad \cB(1-q_2,1),
\]
which will admit some sort of equilibrium. In other words, we want to divide the whole population into three subgroups, corresponding to the lower $100q_1\%$, the middle $100(1-q_1-q_2)\%$ and the upper $100q_2\%$ of the population, where the effectiveness is measured by the benchmark. To do so, let us give a definition of {\it equilibrium state} or {\it global balance}.

\begin{definition}\label{def:equilibrium}
We will say that a {\it global balance} (or {\it equilibrium state}) is achieved in $X$ if
\begin{equation}\label{eq:balance}
\Sigma_{[0,q_1]}=\Sigma_{[q_1,1-q_2]}=\Sigma_{[1-q_2,1]},
\end{equation}
for some $q_1,q_2\in [0,1]$, such that $q_1<q_2$.
\end{definition}
Definition \ref{def:equilibrium} seems to be very intuitive. Indeed, the equality of conditional covariance matrices say that:
\begin{enumerate}
\item The dispersion measured by variance is the same in each subgroup for any coordinate $X_i$ (for $i=1,2,\ldots,n$). In particular the dispersion of the benchmark is the same everywhere.
\item The linear dependance structure, measured by the conditional correlation matrices, is the same in all three subgroups.
\end{enumerate}
The first property creates a natural equilibrium state, as any perturbation leads to irregularity, when the square distance from the average member of each group is considered. The choice of this measure of dispersion seems to be natural, because people awareness of any differences should be high, as  variance (or standard deviation) seems to be the simplest measure of variability.

The second property relate to the linear dependence structure. The equality of correlation matrices imply a natural equilibrium between groups, as people tend to notice the simplest (linear) dependancies first. Any shift between groups will cause dependence instability between them.

In general (i.e. when we loose assumption about normality) the global balance might not exists or strongly depend on initial $\Sigma$, when we consider some family parametrised by covariance matrices.

\section{The 20/60/20 principle}\label{S:206020}
If $X$ is a multivariate normal, it is reasonable to set $q_1=q_2$, due to the symmetry of the Gaussian density. For simplicity we will use $q=q_1=q_2$ for the symmetric case. Thus, we will in fact seek for $q\in (0,0.5)$ such that the conditional covariance matrix for the lower $100q\%$ of the population coincide with the conditional covariance matrices of the middle $100(1-2q)\%$ and upper $100q\%$.

We are now ready to present the main result of this paper. We will show that if $X\sim \cN(\mu,\Sigma)$, then the equilibrium state will be achieved for the unique $q\in (0,0.5)$. This is a statement of Theorem~\ref{th:1}.

\begin{theorem}\label{th:1}
Let $X\sim \cN(\mu,\Sigma)$. Then there exists a unique $q\in (0,0.5)$ such that the global balance in $X$ is achieved, i.e. the equality \eqref{eq:balance} is true for $q=q_1=q_2$. Moreover, the value of $q$ is independent of $\mu$ and $\Sigma$ and the approximate value of $q$ is 0,198089616...
\end{theorem}

The proof of Theorem~\ref{th:1} is surprisingly simple. It is a direct consequence of Lemma~\ref{lm:1} and Lemma~\ref{lm:2}, which we will now present and prove.  Before we do this, let us give a comment on Theorem~\ref{th:1}. It says that if we split the whole population, into three separate groups, then the ratio close to 20-60-20 (and in fact only this ratio), will imply the equality of conditional covariance matrices for all groups, creating a natural equilibrium.
To prove  Theorem~\ref{th:1} we need an analytic formula for conditional covariance structure, given any conditioning Borel set $\cB$ of positive measure. This will be the statement of Lemma \ref{lm:1}.

\begin{lemma}\label{lm:1}
Let $X\sim \cN(\mu,\Sigma)$. Then for any Borel subset $\cB$ of $\bR$ with positive measure,
\[
\Sigma_{\cB}= \Sigma + (D^2[X_1 \mid  X_1 \in \cB] - D^2[X_1]) \beta \beta^T, \]
where 
\[ \beta^T =(\beta_1, \dots , \beta_n), \;\;\; \beta_i= \frac{Cov[X_1,X_i]}{D^2[X_1]}.\]
\end{lemma}

\begin{proof}[Proof of Lemma~\ref{lm:1}]
Being in Gaussian world we can describe each random variable $X_i$ as a combination of the random variable $X_1$ and a random variable $Y_i$ independent of $X_1$.
Indeed, we put for $i=1, \dots , n$
\begin{equation}\label{eq:Yi}
Y_i = X_i - \beta_i X_1, \;\; \mbox{ where } \beta_i= \frac{Cov[X_1,X_i]}{D^2[X_1]}.
\end{equation}
Obviously $\beta_1=1$ and $Y_1=0$.
Since for $i=2, \dots ,n$, the newly defined variable $Y_i$ is uncorrelated with $X_1$, they are independent.\\
Next, we calculate the conditional covariance matrix. Using \eqref{eq:Yi}, we get for $i,j=1,\ldots,n$
\[ Cov[X_i,X_j\mid  X_1 \in \cB] = Cov[ \beta_iX_1 + Y_i, \beta_j X_1 +  Y_j\mid    X_1 \in \cB].\]
Since $Y_i$ and $Y_j$ do not dependent on $X_1$, we get
\[ Cov[Y_i, X_1\mid  X_1 \in \cB] =0= Cov [Y_j,X_1 \mid X_1 \in \cB] ,\]
and 
\[ Cov[Y_i,Y_j \mid  X_1 \in \cB] = Cov[Y_i,Y_j] = Cov[X_i,X_j] - \beta_i\beta_j D^2[X_1].\]
Therefore, we obtain
\[ Cov[X_i,X_j\mid  X_1 \in \cB] = Cov[X_i,X_j] + \beta_i\beta_j (D^2[X_1 \mid  X_1 \in \cB] - D^2[X_1]).\]
Since $\beta_i\beta_j$ is the $i,j$-th entry of the $n \times n$ matrix $\beta \beta^T$, we finish the proof of the lemma.
\end{proof}

From Lemma~\ref{lm:1} we see, that we can parametrise $\Sigma_{\cB}$ in such a way, that it will only depend on the conditional variance of $X_1$. Thus, we only need to show that there exists $q\in (0,0.5)$ such that the (conditional) dispersion of $X_1$ in all three groups, determined by sets $\cB(0,q)$, $\cB(q,1-q)$ and $\cB(1-q,1)$ will coincide.
This will be the statement of Lemma~\ref{lm:2}.
\begin{lemma}\label{lm:2}
Let $X_1\sim \cN(\mu_1,\sigma_{11}^2)$. Then there exist a unique $q\in (0,0.5)$ such that
\[
D^2[X_1\mid  X_1 \in \cB(0,q)]=D^2[X_1 \mid  X_1 \in \cB(q,1-q)]=D^2[X_1 \mid  X_1 \in \cB(1-q,1)].
\]
Moreover,  $q= \Phi(x)$, where $x < 0$ is the unique negative solution of the following equation
\begin{equation}\label{eq:equation}
-x \Phi(x) = \phi(x)(1- 2\Phi(x)),
\end{equation}
where $\phi$ and $\Phi$ denote the density and distribution function of standard normal, respectively. The approximate value of $q$ is 19,8089616....
\end{lemma}

\begin{proof}[Proof of Lemma~\ref{lm:2}]
Without any loss of generality we may assume that $X_1$ has the standard normal distribution $\cN(0,1)$. Indeed, for $X_1^{st}= \frac{X_1-\mu_1}{\sigma_{11}}$, and $q_1,q_2\in [0,1]$, such that $q_1<q_2$, we get
\[
D^2\big[X_1 \mid   H_{1}(X_1) \in [q_1,q_2]\big] = D^2\big[\sigma_{11} X_1^{st} + \mu_1 \mid  \Phi (X_1^{st}) \in [q_1,q_2] \big]= \sigma_{11}^2 D^2\big[X_1^{st} \mid  \Phi (X_1^{st}) \in [q_1,q_2]\big].
\]

To proceed, we need to compute the first two moments of the truncated normal distribution of $X_{1}$. For transparency, we will show full proofs (compare \cite[Section 13.10.1]{JohKotBal1994}).

Let us calculate the conditional expectations $E[X_{1}\mid X_{1}<x]$ and $E[X_1 \mid  x < X_1 < -x]$ for any fixed $x \in (-\infty,0)$. Since $\phi'(x)= -x \phi(x)$, we get
\begin{align*}
E[X_1 \mid  X_1  <x] & = \frac{1}{\Phi(x)} \int_{-\infty}^x \xi \phi(\xi) d\xi = \frac{1}{\Phi (x)} (- \phi(\xi))|_{-\infty}^x = - \frac{\phi(x)}{\Phi(x)},\\
E[X_1 \mid  x < X_1 < -x] & = 0.
\end{align*}
To get the corresponding second moments we integrate by parts.
\begin{eqnarray*}
E[X_1^2 \mid  X_1  <x] &=& \frac{1}{\Phi(x)} \int_{-\infty}^x \xi^2 \phi(\xi) d\xi = \frac{1}{\Phi(x)} \left(- \xi \phi(\xi))|_{-\infty}^x  + \int_{-\infty}^x \phi(\xi)d\xi \right)\\
&=& \frac{1}{\Phi(x)} \left( -x\phi(x) + \Phi(x) \right) = 1 - \frac{x \phi(x)}{\Phi(x)},\\
 E[X_1^2 \mid  x <X_1  <-x] &=& \frac{1}{1 -2\Phi(x)} \int_{x}^{-x} \xi^2 \phi(\xi) d\xi = \frac{1}{1-2\Phi(x)} \left(- \xi \phi(\xi))|_{x}^{-x}  + \int_{x}^{-x} \phi(\xi)d\xi \right)\\
&=&
\frac{1}{1-2\Phi(x)} \left( 2x\phi(x) + 1 -2\Phi(x) \right) = 1 + \frac{2x \phi(x)}{1-2\Phi(x)}.
\end{eqnarray*}
Therefore,
\[ D^2[ X_1 \mid  X_1  <x] = 1 - \frac{x \phi(x)}{\Phi(x)} - \frac{\phi(x)^2}{\Phi(x)^2},\]
\[ D^2[ X_1 \mid  x <X_1  <-x] = 1 + \frac{2x \phi(x)}{1-2\Phi(x)}.\]
Since the conditional expected value behaves like  a weighted arithmetic mean, we get that 
$E[X_1 \mid  X_1  <x]$ is strictly increasing in $x$, while $E[X_1^2 \mid  x <X_1  <-x]$ and $E[X_1^2 \mid  X_1  <x]$ are strictly decreasing with respect to $x$.
Consequently, the {\it central} conditional variance $ D^2 [X_1 \mid  x <X_1  <-x]$ is strictly decreasing. Next, we will show that the {\it tail} conditional variance $D^2[ X_1 \mid  X_1  <x]$ is strictly increasing. Indeed, 
\begin{eqnarray*}
 \frac{d}{d x} D^2[ X_1 \mid  X_1  <x]&=& -\frac{\phi(x)}{\Phi(x)} + x^2 \frac{\phi(x)}{\Phi(x)} -x \frac{\phi(x)^2}{\Phi(x)} +2x\frac{\phi(x)^2}{\Phi(x)} + 2 \frac{\phi(x)^2}{\Phi(x)^2}\\
&=& \frac{\phi(x)}{\Phi(x)} \left(x^2-1 +x \frac{\phi(x)}{\Phi(x)} +2 \frac{\phi(x)^2}{\Phi(x)^2}\right)\\
&=& \frac{\phi(x)}{\Phi(x)} \left(\left(x^2-\frac{1}{2}\frac{\phi(x)}{\Phi(x)}\right)^2 + \frac{7}{4}\frac{\phi(x)^2}{\Phi(x)^2} -1 \right) > 0. 
\end{eqnarray*}
The last inequality follows from the fact that since $\frac{\phi(x)}{\Phi(x)}=-E[X_1 \mid  X_1  <x]$ is decreasing and positive, we get
\[ \frac{\phi(x)^2}{\Phi(x)^2} \geq \frac{\phi(0)^2}{\Phi(0)^2}= \frac{2}{\pi} > \frac{4}{7}.\]
Next, note that (compare \cite[Lemma 8.1]{JawPit2014})
\[ \lim_{x\rightarrow -\infty}  D^2[ X_1 \mid  X_1  <x] =0 \;\;\; \mbox{ and }\;\;\;  D^2[ X_1 \mid  X_1  <0]= 1 - \frac{2}{\pi},\]
while 
\[ \lim_{x\rightarrow -\infty}  D^2[ X_1 \mid  x <X_1  <-x] =1 \;\;\; \mbox{ and } \;\;\; \lim_{x\rightarrow 0} D^2[ X_1 \mid x< X_1  <-x]= 0.\]
Hence there exists a unique $x<0$ such that
\[ D^2[ X_1 \mid  X_1  <x] = D^2[ X_1 \mid  x <X_1  <-x].\]
Compare Figure \ref{fig:d3} for visualization.
\begin{figure}[!ht]
\begin{center}
\includegraphics[width=7cm]{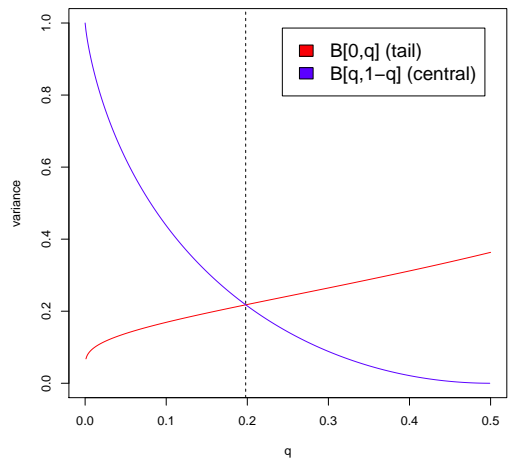}
\end{center}
\vspace{-20pt}
\caption{The graph of conditional tail variance $ D^2[X_1\mid  X_1 \in \cB(0,q)]$ and conditional central variance $D^2[X_1\mid  X_1 \in \cB(q,1-q)]$ as functions of $q\in (0,0.5)$, under the assumption $X_{1}\sim \cN(0,1)$.
}\label{fig:d3}
\end{figure}

\noindent Moreover,
\begin{eqnarray*}
 D^2[ X_1 \mid  X_1  <x] - D^2[ X_1 \mid  x <X_1  <-x]&=& 1- \frac{x\phi(x)}{\Phi(x)} - \frac{\phi(x)^2}{\Phi(x)^2} - 1  - \frac{2x \phi(x)}{1-2\Phi(x)}\\
&=&\frac{\Phi(x)}{\Phi(x)^2 (1-2\Phi(x))} \left( -x \Phi(x) -\phi(x)(1-2\Phi(x))\right),
\end{eqnarray*}
which shows that $x$ is a (negative) solution of equation~\eqref{eq:equation}. Using basic numerical tools we checked that~\eqref{eq:equation} is satisfied for $x\approx-0,8484646848$, for which $\Phi(x)\approx 0,198089615$.
\end{proof}

Theorem~\ref{th:1} provides an illustration to the empirical 20-60-20 rule. In particular we have shown that for any multivariate normal vector, this fixed ratio leads to a global equilibrium state, when dispersion and linear dependance measurement is considered. 
Nevertheless, please note, that the equality of conditional variances does not imply the equality of conditional distributions, as could be seen in Figure 1.
\begin{figure}
\begin{center}
\includegraphics[width=4cm]{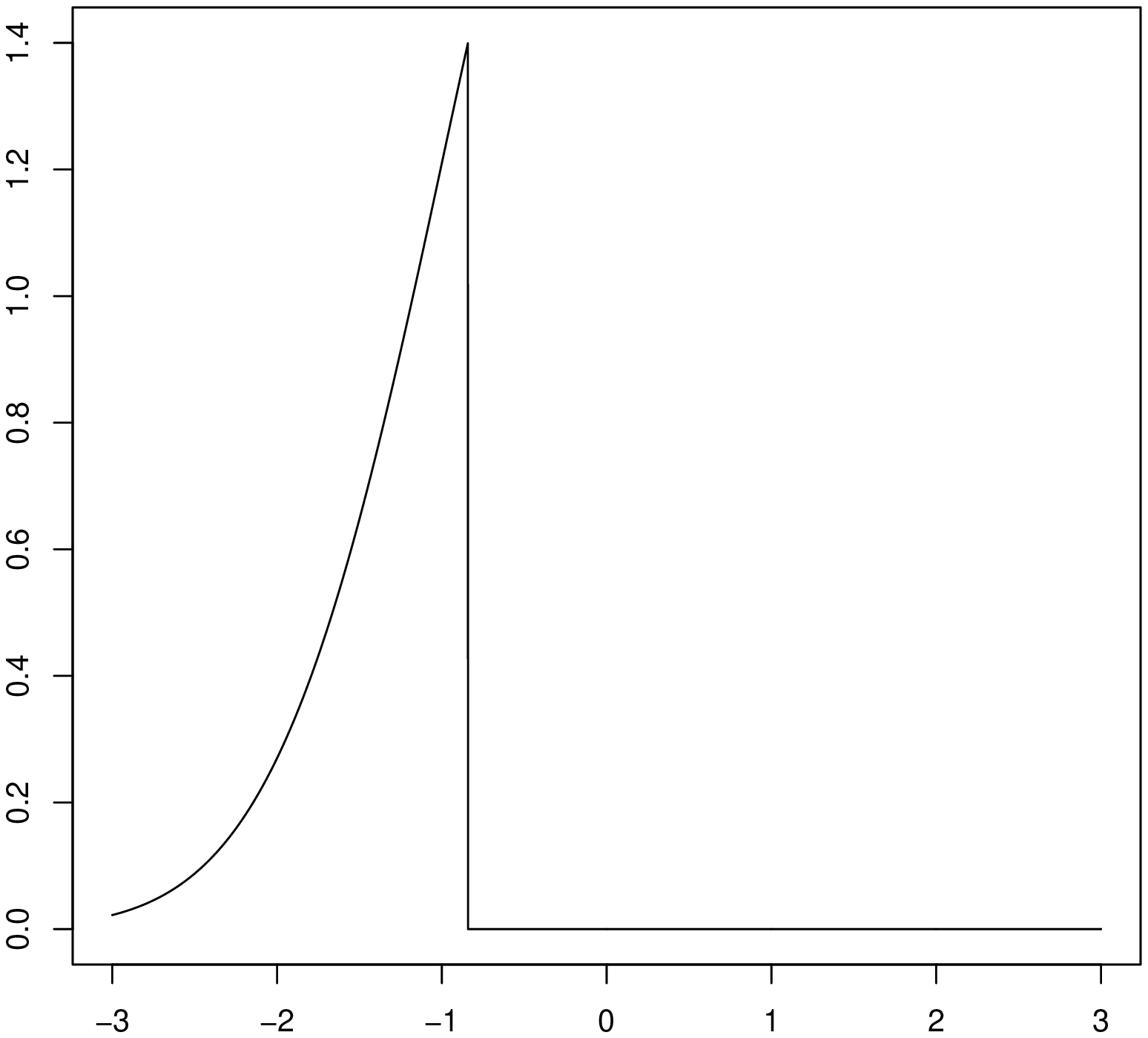}
\includegraphics[width=4cm]{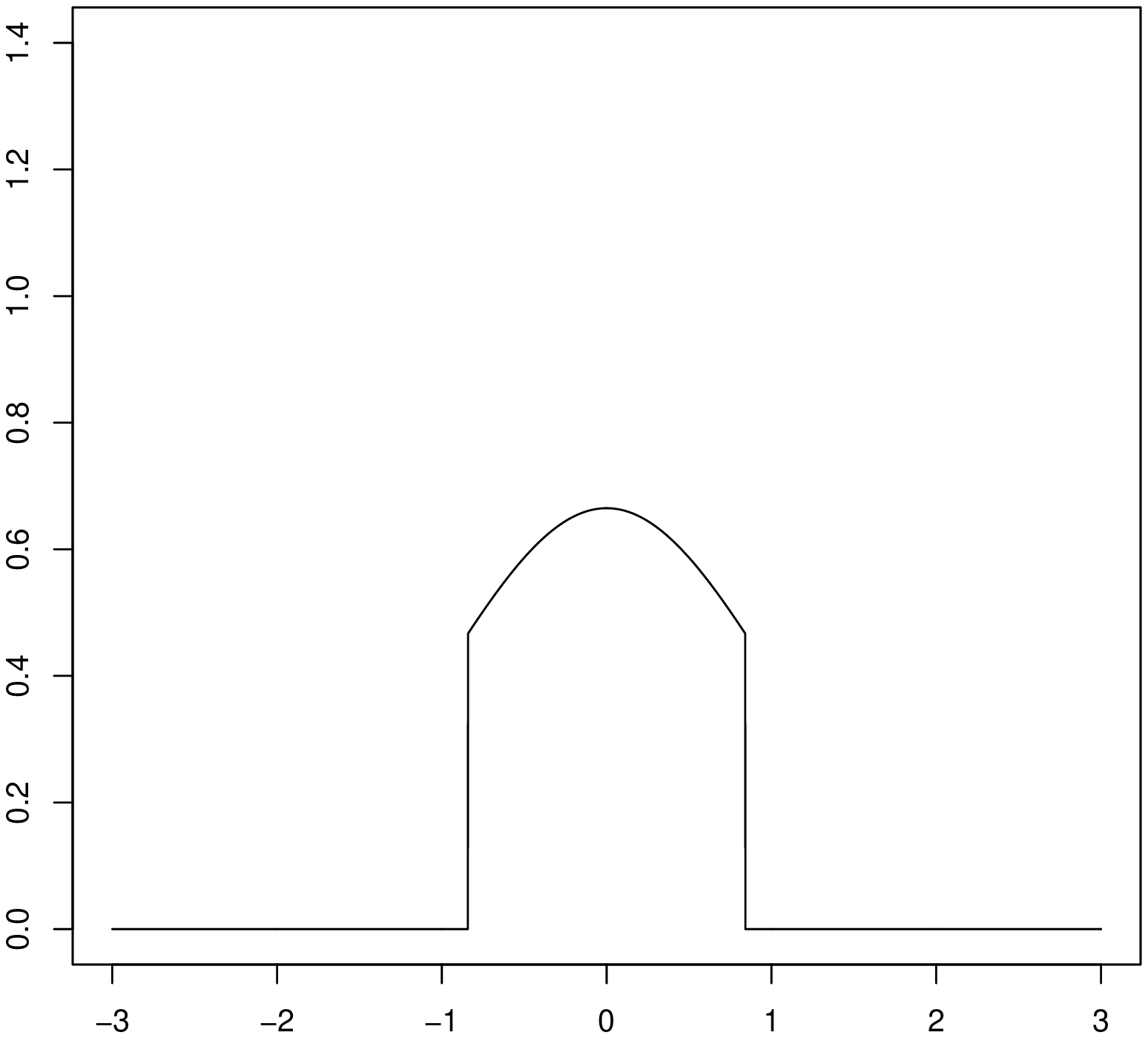}
\includegraphics[width=4cm]{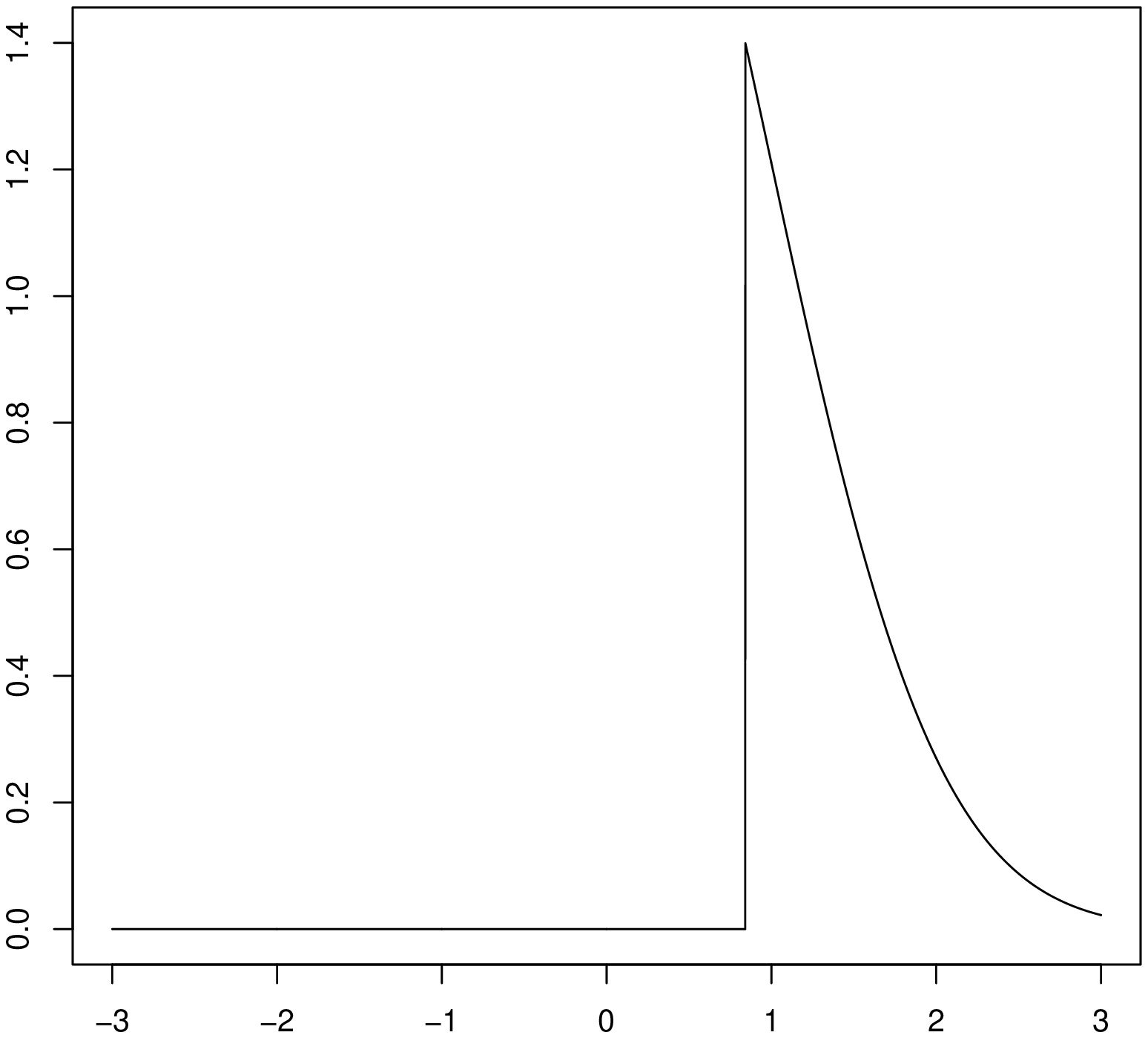}
\end{center}
\vspace{-20pt}
\caption{The conditional density function of the lower 20\%, middle 60\% and upper 20\% of the standard normal distribution. The conditional variances for all three cases coincide.} \label{fig:d}
\end{figure}

Also, while linear dependance structure will be the same, the overall dependance in each subgroup, measured e.g. by the copula function \cite{Nel2007}, will be different. Indeed, for example it seems to be unwise to require the dependance structure in the best group, to coincide with the dependance structure in the average group. See Figure 2 for an illustrative example.

\begin{remark}\label{rem:ind}
The equilibrium level $q$ calculated in Lemma \ref{lm:2} depends neither on $\mu$ nor $\Sigma$. Therefore, if we consider correlation matrices instead of covariance matrices in \eqref{eq:balance}, then the optimal value of $q$ from Theorem~\ref{th:1} will also imply the corresponding equilibrium state, for correlation matrices.\footnote{Please note we need additional assumption that $X_1$ is not independent of $(X_{2},\ldots,X_{n})$ as otherwise any $q\in (0,0.5)$ will satisfy \eqref{eq:balance} for correlation matrices instead of covariance matrices.}
\end{remark}

\begin{remark}
The value $\| \Sigma_{[0,q]}-\Sigma_{[q,1-q]}\|$, for $q\approx0.198$ and some arbitrary matrix norm (e.g. Frobenius norm) might be used to test how far $X$ is from a multivariate normal distribution. This test is particularly important, as it shows the impact of the tails on the central part of the distribution, as usually (for empirical data) the dependence (correlation) structure in the tails significantly increases, revealing non-normality.
\end{remark}

\begin{remark}
We can also consider more than three states, when clustering the population (e.g. having 5 states we might relate to them as critical, bad, normal, good and outstanding performance, based on selected benchmark). The ratios, which imply equilibrium state (similar to the one from Definition~\ref{def:equilibrium}) for 5 and 7 different states are close to 
\[
0.027/0.243/0.460/0.243/0.027\quad\textrm{and}\quad0.004/0.058/0.246/0.384/0.246/0.058/0.004,
\]
respectively. Those values could be easily computed using results from Lemma~\ref{lm:1} and Lemma~\ref{lm:2}.
\end{remark}

\begin{figure}
\begin{center}
\includegraphics[width=4cm]{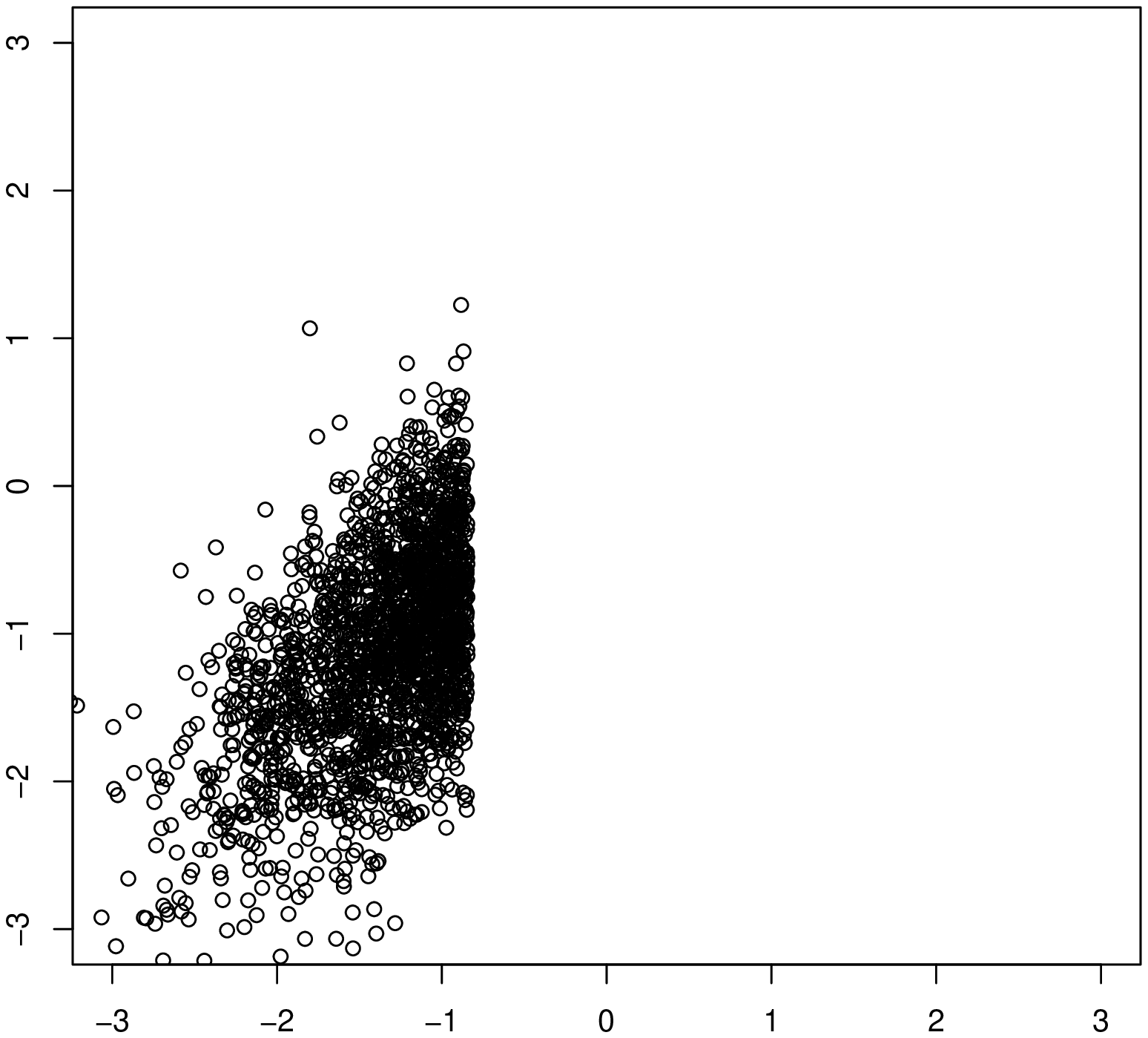}
\includegraphics[width=4cm]{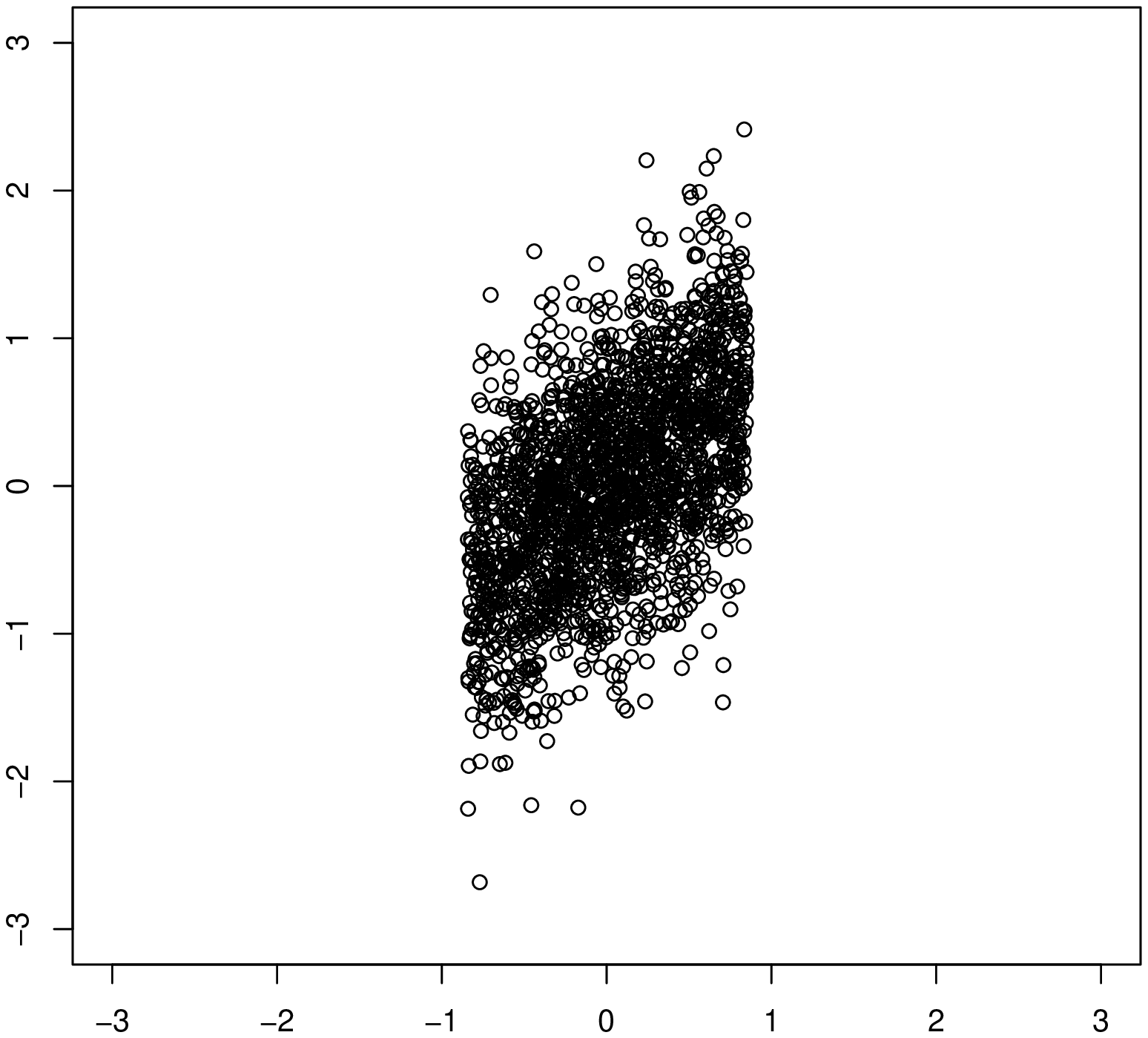}
\includegraphics[width=4cm]{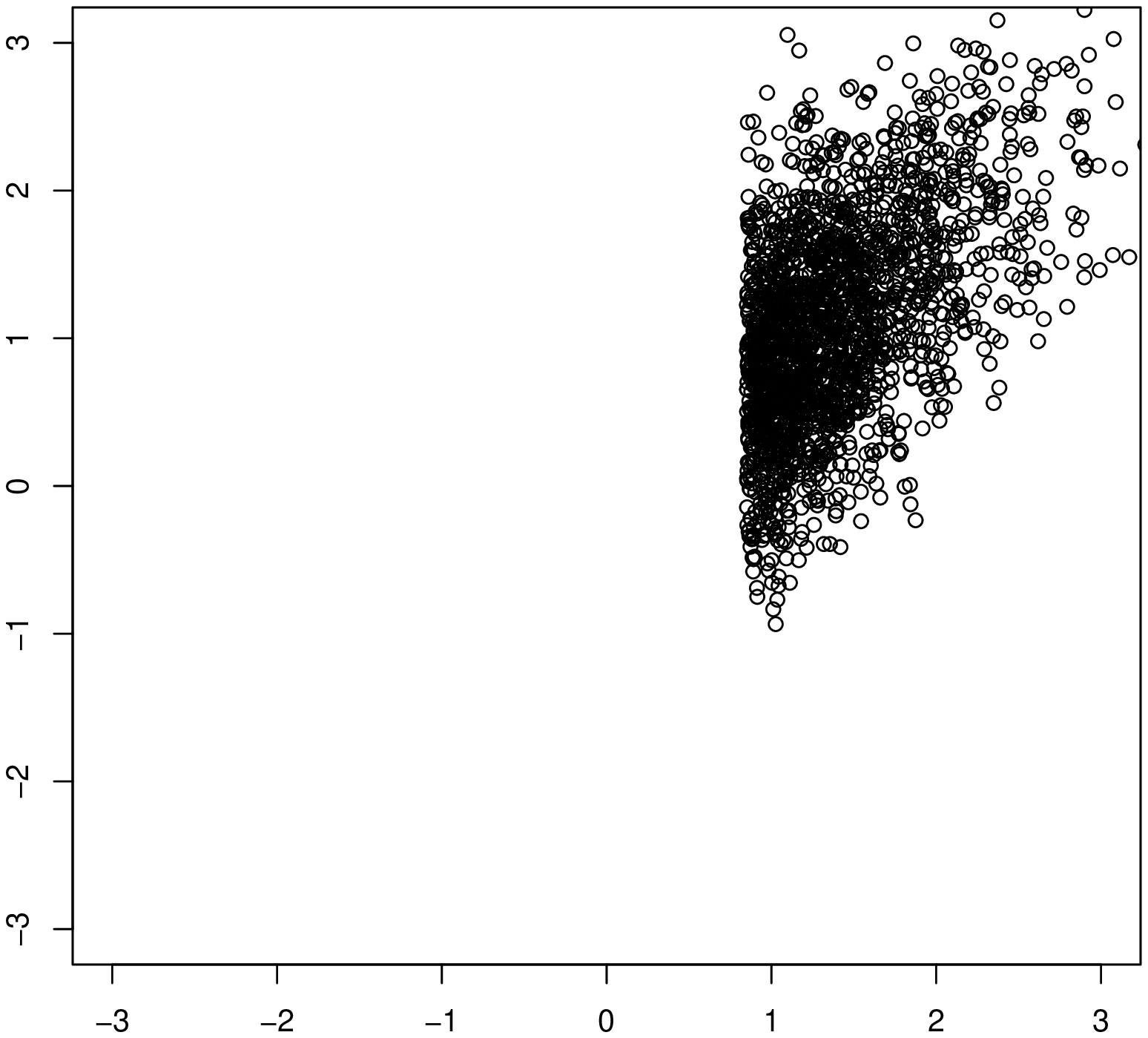}\\
\includegraphics[width=4cm]{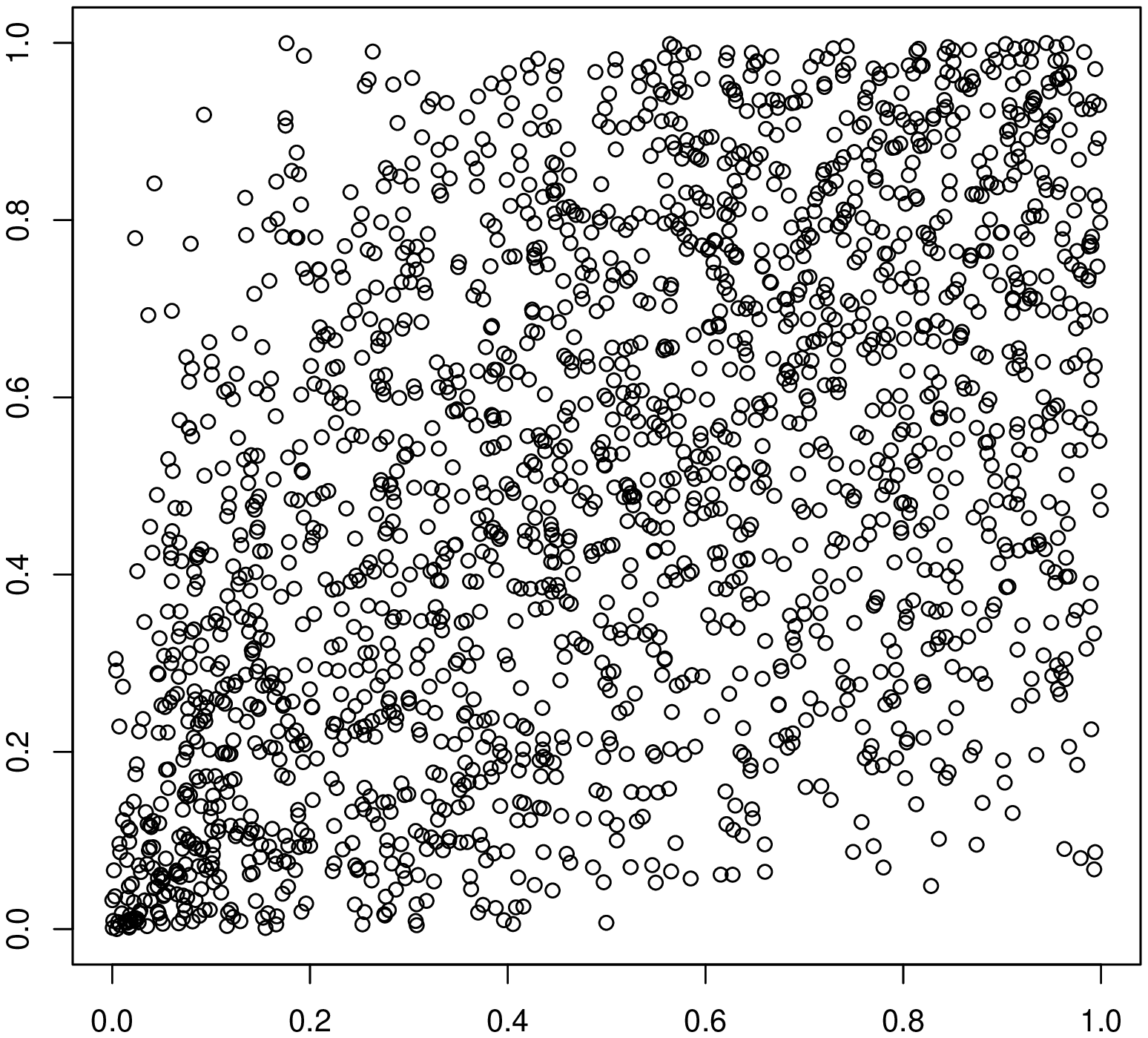}
\includegraphics[width=4cm]{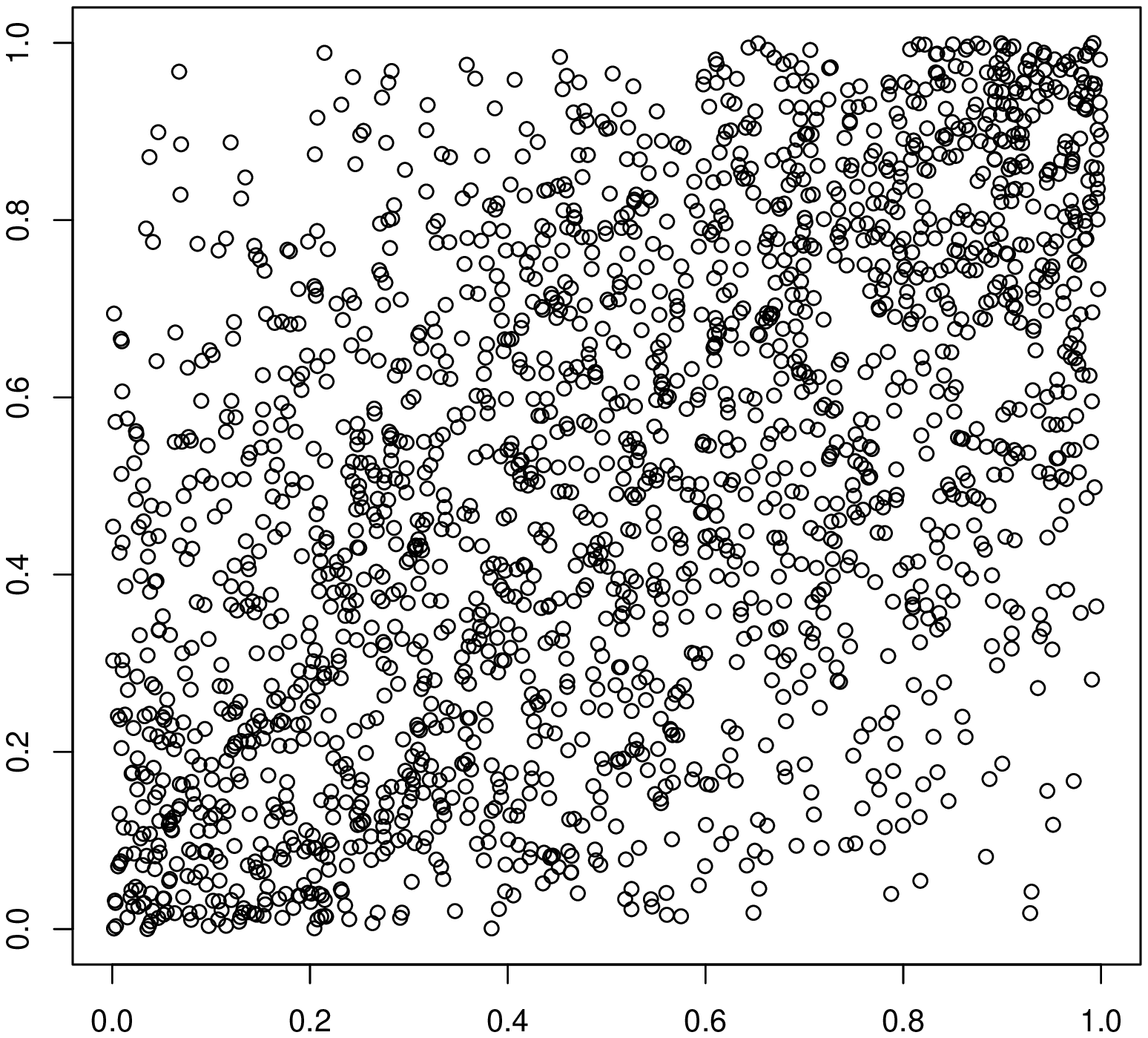}
\includegraphics[width=4cm]{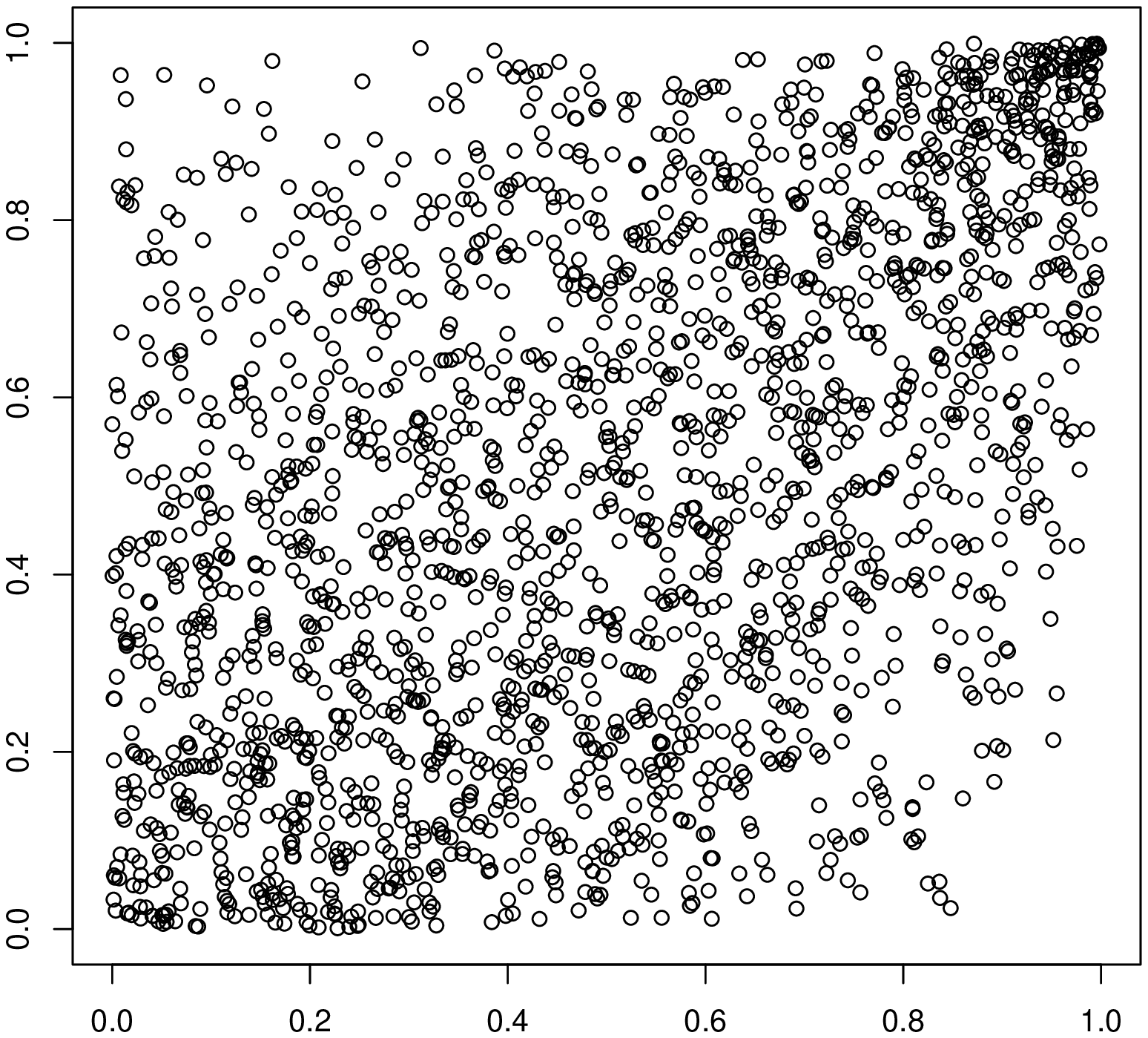}
\end{center}
\vspace{-20pt}
\caption{The conditional samples (upper row) and their conditional copula functions (lower row) from the bivariate normal with $\mu=(0,0)$ and $\Sigma=\sigma_{ij}$, where $\sigma_{11}=\sigma_{22}=1$ and $\sigma_{12}=\sigma_{21}=0.8$. The conditioning is based on the the first coordinate and relates to the lower $20\%$, middle 60\% and upper 20\% of the whole population.} \label{fig:d2}
\end{figure}

\section{Equilibrium for monotonic dependance}\label{S:different.norm}
In the definition of the equilibrium state (Definition~\ref{def:equilibrium}) we have in fact measured the distance between conditional covariance matrices to compare the variability and linear dependance structure between the groups. As explained in Remark~\ref{rem:ind}, one could use conditional correlation matrices instead of covariance matrices and focus on the comparison of the linear dependance structure. Of course there are also other measures of dependance, which could be used to reformulate Definition~\ref{def:equilibrium}.

Among most popular ones are so called measures of concordance, where Kendall $\tau$ and Spearman $\rho$ are usually picked representatives for two dimensional case (see \cite[Section 5]{Nel2007} for more details). Instead of measuring the linear dependence, they focus on the monotone dependence, being invariant to any strictly monotone transform of a random variable (note that correlation is only invariant wrt. positive linear transformation).

Thus, instead of covariance matrices $\Sigma_{[0,q]}$, $\Sigma_{[q,1-q]}$ and $\Sigma_{[1-q,1]}$ in \eqref{eq:balance} we can consider the corresponding matrices of conditional Kendall $\tau$ and conditional Spearman $\rho$, denoted by  $\Sigma^{\tau}_{[0,q]}$, $\Sigma^{\tau}_{[q,1-q]}$, $\Sigma^{\tau}_{[1-q,1]}$ and $\Sigma^{\rho}_{[0,q]}$, $\Sigma^{\rho}_{[q,1-q]}$, $\Sigma^{\rho}_{[1-q,1]}$, respectively. For comparison, we will also consider conditional correlation matrices, for which we shall use notation $\Sigma^{r}_{[0,q]}$, $\Sigma^{r}_{[q,1-q]}$ and $\Sigma^{r}_{[1-q,q]}$.

Unfortunately, the analog of Theorem~\ref{th:1} is not true, if we substitute covariance matrices with the Spearman $\rho$ or Kendall $\tau$ matrices in Definition~\ref{eq:balance}. Because of that we need different kind of notation for the equilibrium state, as stated in Definition~\ref{def:equilibrium2}.

\begin{definition}\label{def:equilibrium2}
Let us assume that $X$ is symmetric\footnote{i.e. $X$ is symmetric wrt. $E[X]=(E[X_{1}],\ldots,E[X_{n}])$; note that it implies that $\Sigma_{[0,q]}=\Sigma_{[1-q,1]}$ for any $q\in (0,0.5)$.} and let $\kappa\in \{r,\rho,\tau\}$\footnote{This will relate to the conditional correlation matrices, Spearman $\rho$ matrices or Kendall $\tau$ matrices, respectively.}. We will say that a {\it quasi-global balance} (or {\it quasi-equilibrium state}) is achieved in $X$ for $\kappa$ and $\hat{q}\in [0,1]$ if
\begin{equation}\label{eq:balance2}
\| \Sigma^{\kappa}_{[0,\hat{q}]}-\Sigma^{\kappa}_{[\hat{q},1-\hat{q}]}\|_{\textrm{F}}=\inf_{q\in (0,0.5)} \| \Sigma^{\kappa}_{[0,q]}-\Sigma^{\kappa}_{[q,1-q]}\|_{\textrm{F}}.
\end{equation}
where $\|\cdot\|_{F}$ is a standard Frobenius matrix norm given by 
\[
\|A \|_{F}:= tr \, AA^T=\sqrt{\sum_{i=1}^{n}\sum_{j=1}^{n}|a_{ij}|^{2}},
\]
for any $n$-dimensional matrix $A=\{a_{ij}\}_{i,j=1,\ldots,n}$.

Similarly as in Definition~\ref{def:equilibrium}, we will say that a {\it global balance} (or {\it equilibrium state}) is achieved in $X$ for $\kappa$ and $\hat{q}\in [0,1]$ if the value in \eqref{eq:balance2} is equal to 0.
\end{definition}

For transparency, we will write
\begin{align}
\hat{q}^{r}=\argmin_{q\in (0,0.5)} \| \Sigma^{r}_{[0,q]}-\Sigma^{r}_{[q,1-q]}\|_{\textrm{F}}\label{eq:qr},\\
\hat{q}^{\tau}=\argmin_{q\in (0,0.5)} \| \Sigma^{\tau}_{[0,q]}-\Sigma^{\tau}_{[q,1-q]}\|_{\textrm{F}}\label{eq:qtau},\\
\hat{q}^{\rho}=\argmin_{q\in (0,0.5)} \| \Sigma^{\rho}_{[0,q]}-\Sigma^{\rho}_{[q,1-q]}\|_{\textrm{F}}\label{eq:qrho},
\end{align}
to denote ratios, which imply quasi-equilibrium states given in \eqref{eq:balance2}.\footnote{For 
simplicity, we use $\argmin$ and assume that the (quasi) equilibrium state exists and is unique.}

As expected, for $X\sim \mathcal{N}(\mu,\Sigma)$, the values $\hat{q}^{\tau}$ and $\hat{q}^{\rho}$ also seem to be very close to 0.2, for almost any value of $\mu$ and $\Sigma$. To illustrate this property, we have picked 1000 random covariance matrices $\{\Sigma_{i}\}_{i=1}^{1000}$ for $n=4$\footnote{With additional assumption that correlation coefficients are bigger than 0.2 and smaller than 0.8, to avoid computation problems resulting from independence or comonotonicity, respectively (see also Remark~\ref{rem:ind}). Note also that the sign of correlation coefficient is irrelevant, due to symmetry of $X$, so without loss of generality, we can assume that the correlation matrix is positive. Moreover, the values of $\hat{q}^{\tau}$ and $\hat{q}^{\rho}$ are invariant wrt. $\mu$, so we can set $\mu=0$ without loss of generality.}  and computed the values of functions
\begin{align}
f^{i}_{r}(q) &=\| (\Sigma^{i})^{r}_{[0,q]}-(\Sigma^{i})^{r}_{[q,1-q]}\|_{\textrm{F}},\label{eq:f1}\\
f^{i}_{\tau}(q) &=\| (\Sigma^{i})^{\tau}_{[0,q]}-(\Sigma^{i})^{\tau}_{[q,1-q]}\|_{\textrm{F}},\label{eq:f2}\\
f^{i}_{\rho}(q) &=\| (\Sigma^{i})^{\rho}_{[0,q]}-(\Sigma^{i})^{\rho}_{[q,1-q]}\|_{\textrm{F}}.\label{eq:f3}
\end{align}
To do so, for each $i\in\{1,2,\ldots,1000\}$ we have taken 1.000.000 Monte Carlo sample from $X\sim \cN(0,\Sigma^{i})$ and computed values of \eqref{eq:f1}, \eqref{eq:f2} and \eqref{eq:f3} using 
MC estimates of the corresponding conditional matrices. 
The graphs of $f_{r}^{i}$, $f^{i}_{\tau}$ and $f^{i}_{\rho}$ for $i=1,2,\ldots,50$ are presented in Figure~\ref{fig:FRtype}. In Figure~\ref{fig:densities}, we also present the smoothed histogram function of points $\{\hat{q}^{r}_{i}\}_{i=1}^{1000}$, $\{\hat{q}_{i}^{\tau}\}_{i=1}^{1000}$ and $\{\hat{q}^{\rho}_{i}\}_{i=1}^{1000}$, for which the minimum is attained in \eqref{eq:f1}, \eqref{eq:f2} and \eqref{eq:f3} for $i=1,2,\ldots,1000$.

\begin{figure}[!ht]
\begin{center}
\includegraphics[width=5cm]{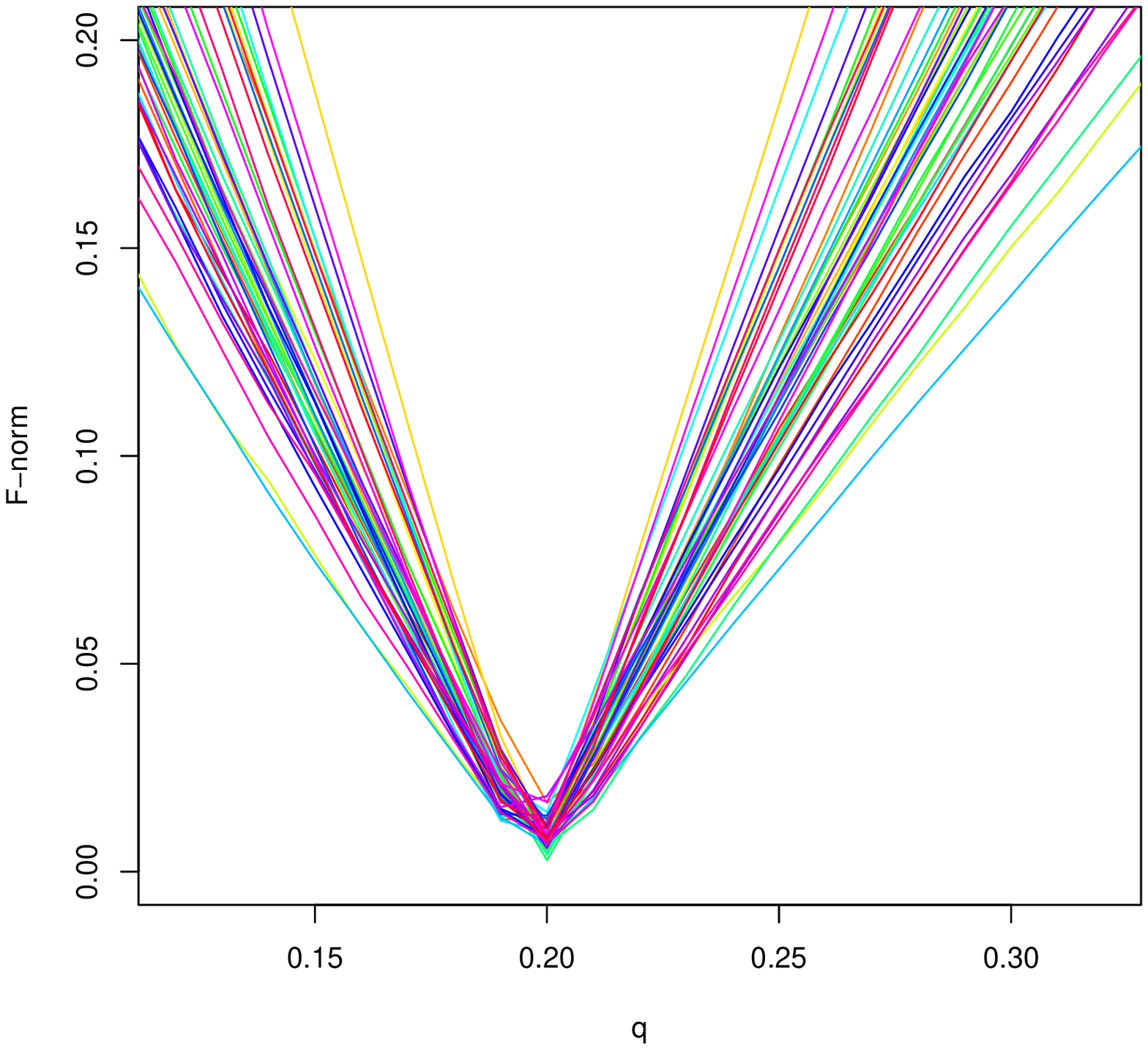}
\includegraphics[width=5cm]{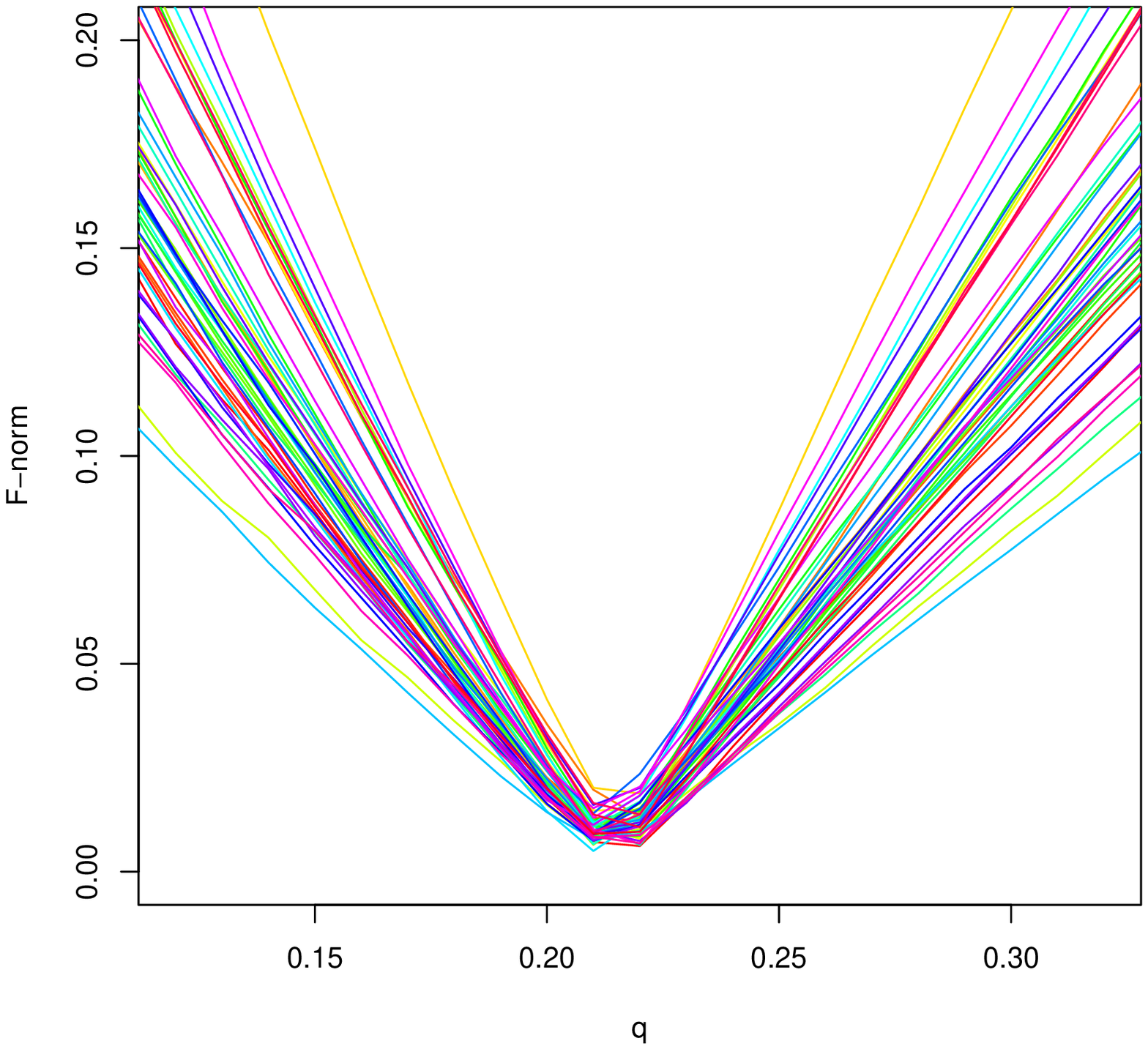}
\includegraphics[width=5cm]{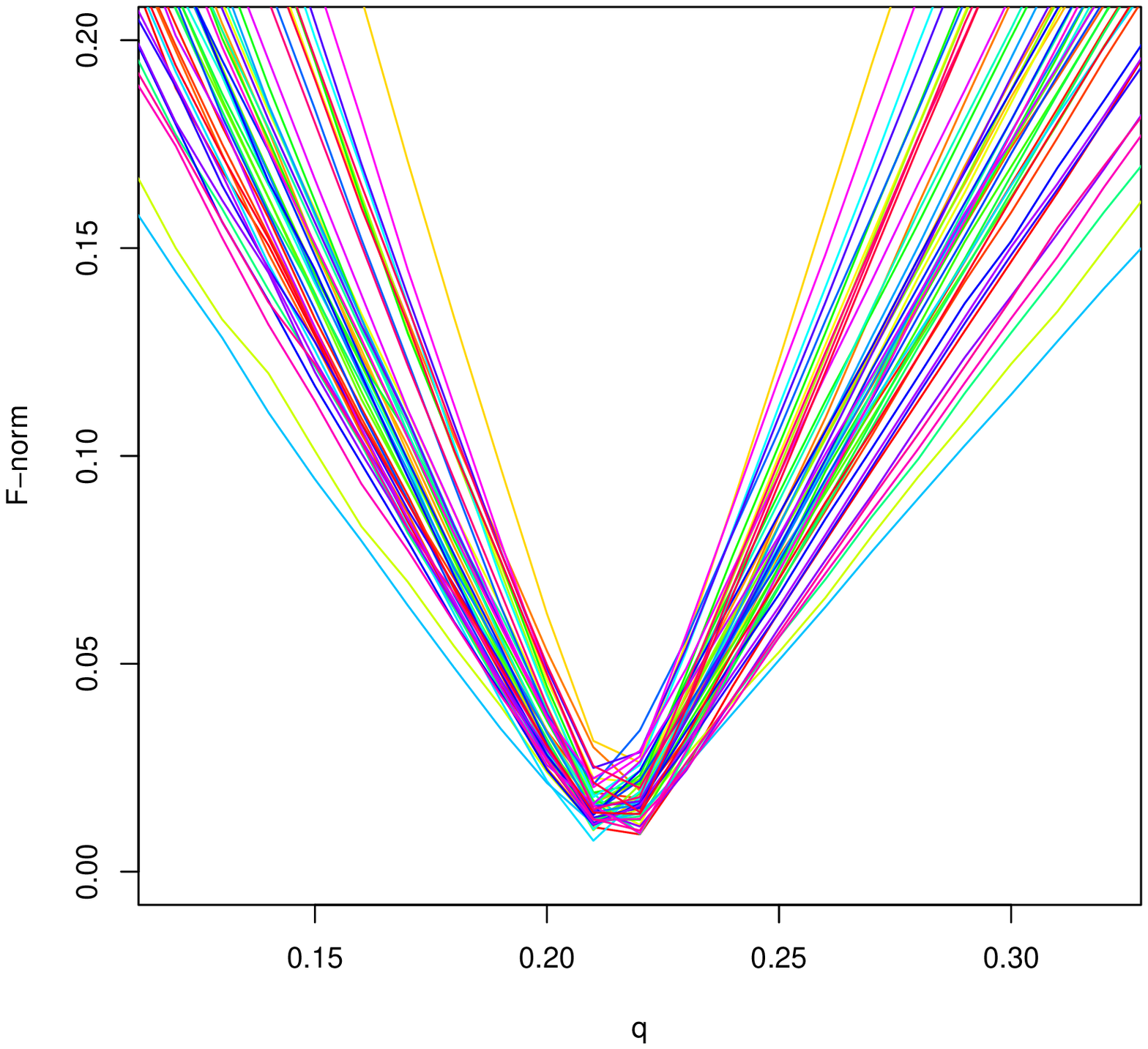}
\end{center}
\vspace{-20pt}
\caption{The graphs of functions $f_{r}^{i}$, $f^{i}_{\tau}$ and $f^{i}_{\rho}$ for $i=1,2,\ldots,50$, computed using 1.000.000 sample from $\cN(0,\Sigma^{i})$ and the corresponding estimates of conditional matrices.} \label{fig:FRtype}
\end{figure}

\begin{figure}[!ht]
\begin{center}
\includegraphics[width=5cm]{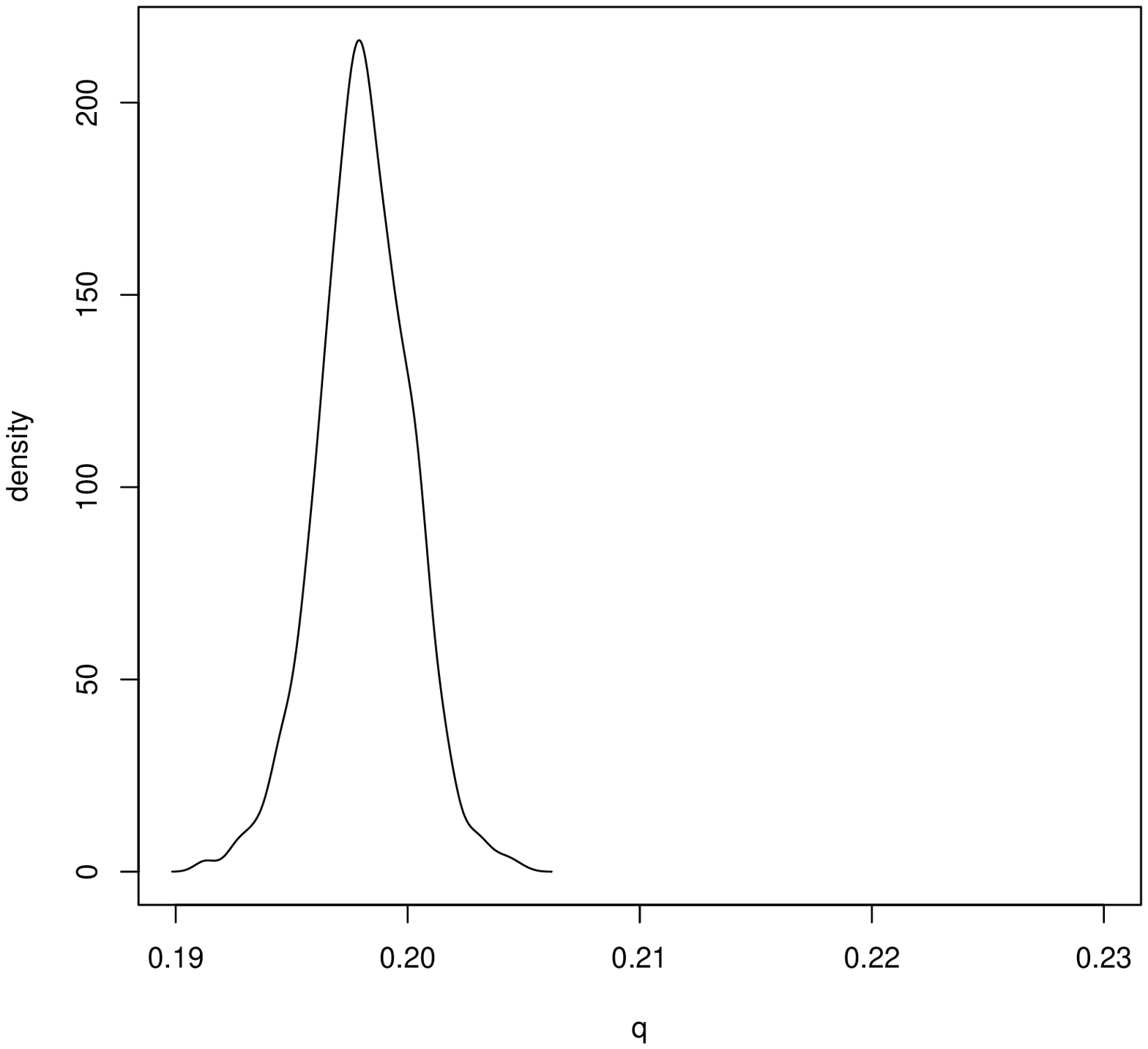}
\includegraphics[width=5cm]{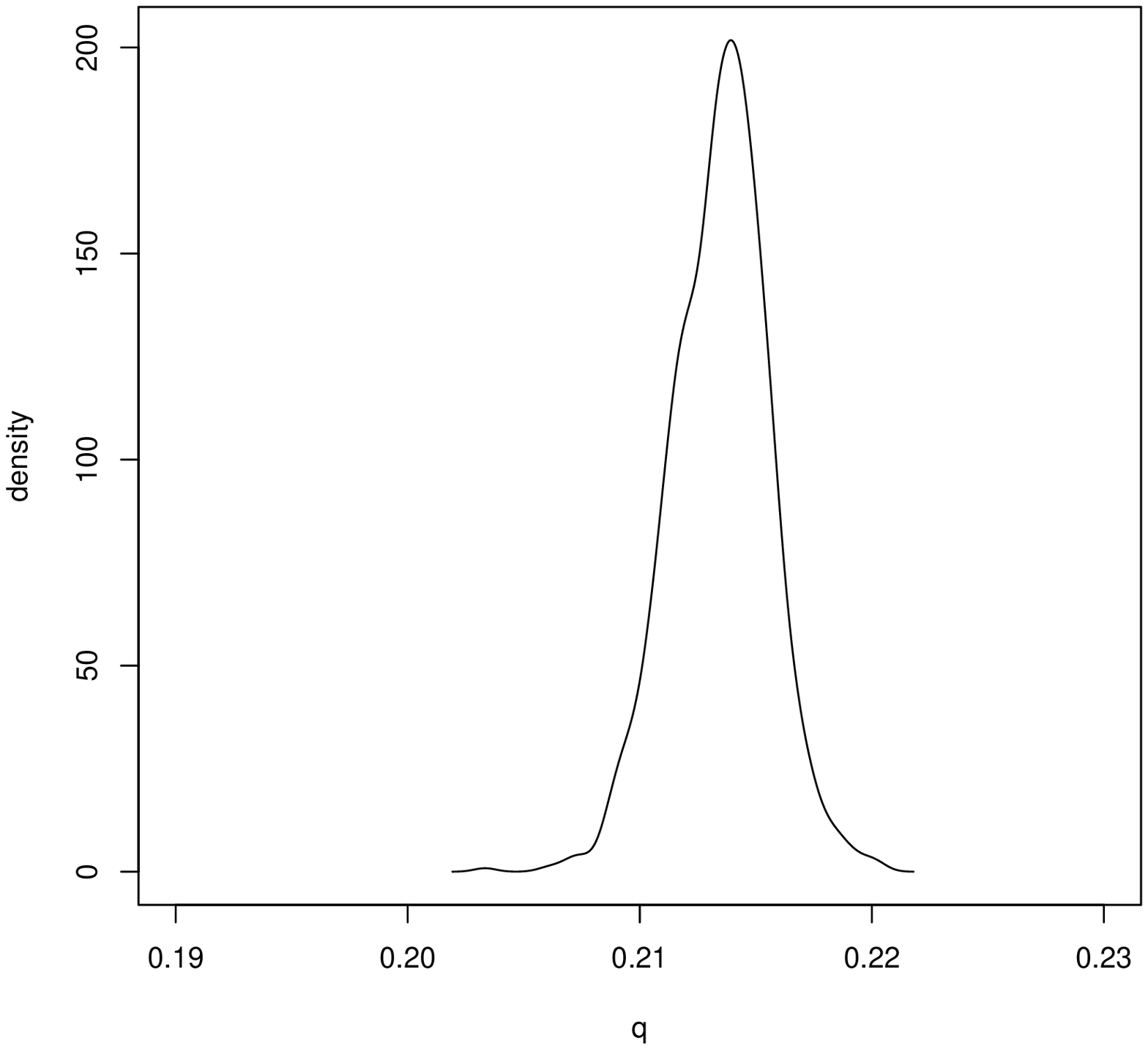}
\includegraphics[width=5cm]{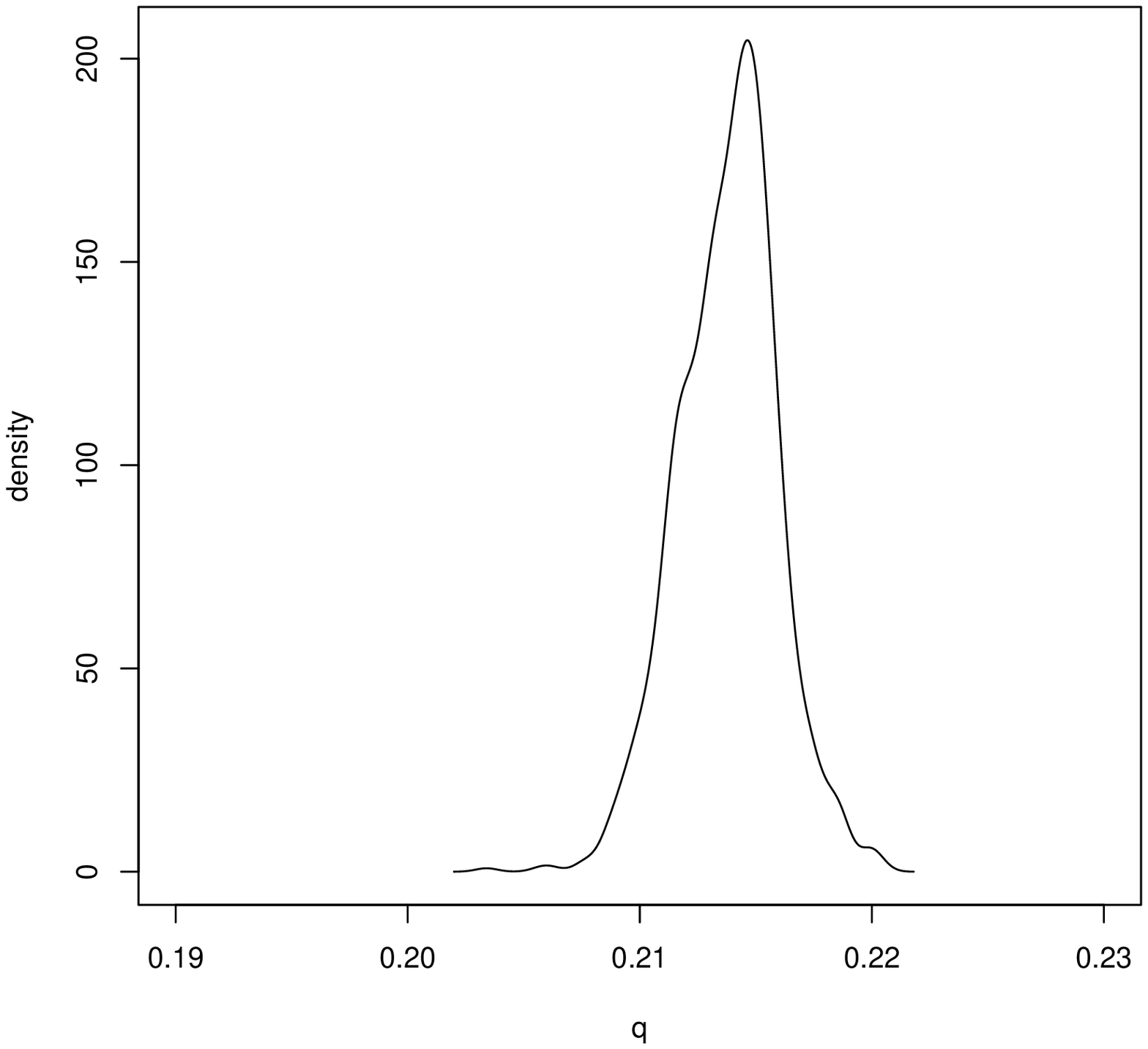}
\end{center}
\vspace{-20pt}
\caption{Monte Carlo density functions constructed using points $\{\hat{q}_{i}\}_{i=1}^{1000}$, $\{\hat{q}_{i}^{\tau}\}_{i=1}^{1000}$ and $\{\hat{q}^{\rho}_{i}\}_{i=1}^{1000}$. For each $i=1,2,\ldots,1000$ a 1.000.000 sample from  $\cN(0,\Sigma^{i})$ was simulated and the corresponding estimates of conditional matrices were used for computations.} \label{fig:densities}
\end{figure}

Unfortunately, in general the values $\hat{q}^{\tau}$ and $\hat{q}^{\rho}$ defined in \eqref{eq:qtau} and \eqref{eq:qrho} are not constant and independent of $\Sigma$. In particular, if the dependance inside $X$ is very strong, e.g. the vector $(X_1,X_{2},\ldots,X_{n})$ is almost comonotone, then the values of $\hat{q}^{\tau}$ and $\hat{q}^{\rho}$ might increase substantially.\footnote{Note that in our numerical example we have assumed that the correlation for any pair is between 0.2 and 0.8, excluding extremal cases.}

To illustrate this property, let us present some theoretical results, involving conditional Spearman $\rho$  {and Kendall $\tau$}. For simplicity, till the end of this subsection, we will assume that $n=2$.

Then, given $X\sim \cN(\mu,\Sigma)$, we know that $\sigma_{12}^{2}=\sigma_{21}^{2}=r\sigma_{11}\sigma_{22}$, where $r\in [-1,1]$ is the correlation between $X_1$ and $X_2$. It is easy to show (see~\cite{JawPit2014}), that both unconditional and conditional values of Spearman $\rho$ {as well as Kendall $\tau$} will depend only on the copula of $X$\footnote{Note that the (conditional) Spearman $\rho$ and Kendall $\tau$ is invariant to any monotone transform of $X_{1}$ or $X_{2}$, and so is the copula function.}, which is parametrised by the correlation coefficient. Thus, without loss of generality, instead of considering all $\mu$ and $\Sigma$, we might assume that
\[
X=(X_1,X_2)\sim \cN(\mu,\Sigma)\quad \textrm{where}\quad \mu=(0,0)\quad \textrm{and}\quad \Sigma=\left(\begin{array}{cc} 1& r \\ r & 1\end{array}\right).
\]
for a fixed $r\in [-1,1]$.

Let $\rho_{[p,q]}(r)$ and $\tau_{[p,q]}(r)$ denote the corresponding conditional Spearman $\rho$ and Kendall $\tau$, given truncation interval $\cB(p,q)$. Note that $ \rho_{[p,q]}(r)$ and $\tau_{[p,q]}(r)$ are odd functions of $r$.

\begin{lemma}\label{lem:odd.rho}
For all $0 \leq p < q \leq 1$ and $r \in (-1,1)$,
\[ \rho_{[p,q]}(- r) = - \rho_{[p,q]}(r) \;\;\; \mbox{ and } \;\;\; \tau_{[p,q]}(-r)=-\tau_{[p,q]}(r). \]
\end{lemma}

\begin{proof}
Before we begin the proof, let us recall some basic facts from the copula theory (cf.~\cite{Nel2007} and references therein). We will use $C^{r}$ to denote the Gaussian copula, with parameter $r\in (-1,1)$, which coincides with the correlation coefficient. Noting, that the copula could be seen as a distribution function (with uniform margins) let us assume that $(U,V)$ is a random vector with distribution $C^{r}$. We will denote by $C^{r}_{[p,q]}$ the copula of the conditional distribution $(U,V)$ under the condition $U\in [p,q]$, where $0\leq  p<q\leq 1$. Due to Sklar's Theorem we get the following description of $C^{r}_{[p,q]}$:
\begin{equation}\label{E:Ccond}
 C^{r}_{[p,q]} \left(u, \frac{C^{r}(q,v) - C^{r}(p,v)}{q-p} \right) = \frac{C^{r}((q-p)u+p,v) - C^{r}(p,v)}{q-p}, \;\;\; u, v\in [0,1].
\end{equation}

Next, it is easy to notice, that the distribution function of $(U,1-V)$ is equal to $C^{-r}$. Hence the Gaussian copulas commute with  flipping, i.e.
\[ C^{-r}(u,v) = u - C^{r}(u,1-v)\quad \textrm{for}\ u,v\in [0,1].
\]
On the other hand the flipping transforms the conditional distribution $(U,V) |_{U \in [p,q]}$ to $(U,1-V) |_{U \in [p,q]}$. Hence
we get
\[ C^{-r}_{[p,q]}(u,v) = u - C^{r}_{[p,q]}(u,1-v).\]
Thus basing on \cite[Theorem 5.1.9]{Nel2007}, we conclude
\[ \rho_{[p,q]}(- r) = -\rho_{[p,q]}(r),\]
\[ \tau_{[p,q]}(- r) = -\tau_{[p,q]}(r),\]
\end{proof}

We recall that the Spearman $\rho$ and Kendall $\tau$ of the conditional copula $C^{r}_{[p,q]}$ are given by  formulas:
\begin{eqnarray}
\rho_{[p,q]}(r) &=&\rho(C^{r}_{[p,q]})= -3 +12 \int_0^1 \int_0^1 C^{r}_{[p,q]}(u,v) \d u\d v,\\
\tau_{[p,q]}(r) &=&\tau(C^{r}_{[p,q]})= -1 +4  \iint_{[0,1]^2} C^{r}_{[p,q]}(u,v) \d C^{r}_{[p,q]}(u,v).
\end{eqnarray}
To describe their behaviour for small $r$ we will  need their Taylor expansions with respect to $r$.
\begin{proposition}\label{P:Taylor}
For a fixed $p,q\in (0,1)$ ($p<q$) and $r\in (-1,1)$, such that $r$ is close to 0, we get
\begin{eqnarray}
\rho_{[p,q]}(r) &=& r \frac{3}{(q-p)^2 \pi} \left( \Phi(\sqrt{2} x_2) - \Phi(\sqrt{2} x_1) - (q-p) \sqrt{\pi}(\varphi(x_1)+ \varphi(x_2)) \right) + O(r^3),\\
\tau_{[p,q]}(r) &=& \frac{2}{3} \rho_{[p,q]}(r) + O(r^3).
\end{eqnarray}
where $x_1=\Phi^{-1}(p)$ and $x_2=\Phi^{-1}(q)$.
\end{proposition}

\begin{proof}
We will use notation similar to the one introduced in Lemma~\ref{lem:odd.rho}. The proof will be based on two facts.
First, for $r=0$ both $C$ and $C_{[p,q]}$ are equal to product copula $\Pi(u,v):=uv$, i.e.
\[C^{0}(u,v)=uv= C^{0}_{[p,q]}(u,v).\]
Second, the derivative of the distribution function of a bivariate Gaussian distribution having standardised margins with respect to the parameter $r$  is equal to its density, which implies
\[
\frac{\partial C^{r}(u,v)}{\partial r} =\frac{1}{2\pi\sqrt{1-r^2}}\exp\Big(-\frac{\Phi^{-1}(u)^2+\Phi^{-1}(v)^2-2r\Phi^{-1}(u)\Phi^{-1}(v)}{2(1-r^2)}\Big).
\]

We calculate the Taylor expansion of $\rho_{[p,q]}(r)$ at $r=0$.
\[ \rho_{[p,q]}(0) = \rho(\Pi) =0.\]
\[ \frac{\partial \rho_{[p,q]}(0)}{\partial r} = 12 \int_0^1 \int_0^1 \frac{\partial C^{0}_{[p,q]}(u,v)}{\partial r}\d u \d v .\]
 The derivative of $C^r_{[p,q]}$ will be calculated in two steps. First we differentiate formula (\ref{E:Ccond}). We get
 \begin{eqnarray*}
&&\frac{\partial}{\partial r} C^{r}_{[p,q]} \left(u, \frac{C^{r}(q,v) - C^{r}(p,v)}{q-p} \right)\\ 
&+& \partial_2 C^{r}_{[p,q]} \left(u, \frac{C^{r}(q,v) - C^{r}(p,v)}{q-p} \right) \frac{1}{q-p} \left( \frac{\partial C^{r}(q, v)}{\partial r} - \frac{\partial C^{r}(p, v)}{\partial r} \right)\\
&=& 
\frac{1}{q-p} \left( \frac{\partial C^{r}((q-p)u+p, v)}{\partial r} - \frac{\partial C^{r}(p, v)}{\partial r} \right).
\end{eqnarray*}
Next, setting $r=0$, we obtain
\[
\frac{\partial}{\partial r} C^{0}_{[p,q]} \left(u, v \right) = 
\frac{1}{q-p} \left( \varphi \left( \Phi^{-1} ((q-p)u+p)\right) - u \varphi(\Phi^{-1}(q)) - (1-u) \varphi(\Phi^{-1}(p)) \right) \varphi(\Phi^{-1}(v))
\]
Finally, we get
\begin{eqnarray*}
\frac{\partial \rho_{q,p}(0)}{\partial r} 
&=& 
\frac{12}{q-p} \int_0^1 \int_0^1  \left( \varphi( \Phi^{-1} ((q-p)u+p)) - u \varphi(\Phi^{-1}(q)) - (1-u) \varphi(\Phi^{-1}(p)) \right) \varphi(\Phi^{-1}(v))\d u\d v\\
&=& \frac{12}{q-p} \int_0^1 \left( \varphi( \Phi^{-1} ((q-p)u+p)) - u \varphi(\Phi^{-1}(q)) - (1-u) \varphi(\Phi^{-1}(p)) \right) \d u
\int_0^1 \varphi(\Phi^{-1}(v)) \d v\\
&=& \frac{12}{q-p} \frac{1}{2 \sqrt{\pi}} \left( \frac{1}{q-p} \frac{1}{2 \sqrt{\pi}} (\Phi( \sqrt{2} \Phi^{-1}(q))- \Phi( \sqrt{2} \Phi^{-1}(p))) \right.
- \left.\frac{1}{2} ( \varphi(\Phi^{-1}(q)) +\varphi(\Phi^{-1}(p)))\right).
\end{eqnarray*}

The proof of the Kendall $\tau$ case follows from the symmetry
\[ \iint C_1 \d C_2 = \iint C_2 \d C_1 .\]
We have
\[ \frac{\partial \tau_{q,p}(r)}{\partial r} =  4 \frac{\partial }{\partial r} \iint_{[0,1]^2} C^{r}_{[p,q]}(u,v) \d C^{r}_{[p,q]}(u,v)
= 8 \iint_{[0,1]^2} \frac{\partial }{\partial r}C^{r}_{[p,q]}(u,v) \d C^{r}_{[p,q]}(u,v).
\]
Setting $r=0$ we get
\[ \frac{\partial \tau_{q,p}(0)}{\partial r} 
= 8 \iint_{[0,1]^2} \frac{\partial }{\partial r}C^{0}_{[p,q]}(u,v) \d C^{0}_{[p,q]}(u,v)= 8 \iint_{[0,1]^2} \frac{\partial }{\partial r}C^{0}_{[p,q]}(u,v) \d u \d v= \frac{8}{12} \frac{\partial \rho_{q,p}(0)}{\partial r} .
\]

\end{proof}
{For $\kappa$ denoting either $\rho$ or $\tau$}, using Proposition~\ref{P:Taylor}, we are now ready to compare values of $\kappa_{[0,q]}(r)$ and $\kappa_{[q,1-q]}(r)$, changing both $q \in (0,0.5)$ and $r\in (-1,1)$. Note that for $n=2$ the equilibrium state corresponding to~\eqref{eq:qrho} is achieved, if and only if $\kappa_{[0,q]}(r) - \kappa_{[q,1-q]}(r)=0$.
In \cite[Theorems 4.1 and 4.4]{JawPit2014}, it was shown that for any fixed $r>0$, the conditional copulas $C^r_{[0,q]}$ are increasing in $q$ while 
$C^r_{[q,1-q]}$ are decreasing in $q$. Hence
the differences
\begin{equation}\label{eq:diff.sp.rho}
\Delta_\rho(q,r)=\rho_{[0,q]}(r) - \rho_{[q,1-q]}(r) \;\;\; \mbox{ and }\;\;\; \Delta_\tau (q,r)=\tau_{[0,q]}(r) - \tau_{[q,1-q]}(r)
\end{equation}
are strictly increasing in $q$ and changing the sign. 
Using Lemma~\ref{lem:odd.rho} we know, that for each $r\in (-1,1)$, such that $r\neq 0$, there exists exactly one $q\in (0,0.5)$ for which 
$\Delta_\rho(q,r)=0$ and one $q\in (0,0.5)$ for which 
$\Delta_\tau(q,r)=0$.
Let
\[A_\kappa\colon (-1,1) \rightarrow (0,0.5), \;\;\; \kappa=\rho,\tau,\]
be a function, which assigns appropriate $q$ for any $r\neq 0$, and let $A_{\kappa}(0)=\liminf_{t\to 0}A_\kappa(t)$\footnote{Note, that for $r=0$, any $q\in (0,0.5)$ implies equilibrium state, the reason we define $A(0)$ in that way.}. We will now show that the graphs of $A_\rho$ and $A_{\tau}$ are orthogonal to the line $r=0$.
\begin{theorem}\label{th:q.orthogonal}
For $r$ close to 0, we get
\[A_\rho(r)=A_\tau(r)+O(r^2) = q^\ast + O(r^2),\]
where $q^\ast\approx 0,2132413$ is a solution of the following equation
\[(1-4q+6q^2)  \Phi(\sqrt{2} \Phi^{-1}(q)) - q(1-6q+8q^2)\sqrt{\pi} \varphi(\Phi^{-1}(q)) -q^2=0.\]
\end{theorem}

\begin{proof}
If $r=0$, then for any $q\in (0,0.5)$, we get that \eqref{eq:diff.sp.rho} is equal to 0, so for clarity we might set $A_\kappa(0)=q^{*}$. Using Lemma~\ref{lem:odd.rho}, without loss of generality, we might assume that $r>0$. Due to Proposition \ref{P:Taylor}, for small $r$, we get
\begin{align*}
\rho_{[0,q]}(r) - \rho_{[q,1-q]}(r) &= \phantom{-} r\frac{3}{\pi} q^{-2} \left(\Phi(\sqrt{2}\Phi^{-1}(q)) - q\sqrt{\pi}\varphi(\Phi^{-1}(q))\right) +O(r^3)\\
&\phantom{=} - r\frac{3}{\pi}(1-2q)^2\left( 1 -2 \Phi(\sqrt{2}\Phi^{-1}(q)) -2(1-2q)\sqrt{\pi} \varphi(\Phi^{-1}(q))\right) +O(r^3)\\
&=\phantom{+}r \frac{3}{\pi q^2 (1-2q)^2} \left( (1-4q+6q^2)  \Phi(\sqrt{2} \Phi^{-1}(q)) - q(1-6q+8q^2)\sqrt{\pi} \varphi(\Phi^{-1}(q)) -q^2 \right)\\
&\phantom{=} +O(r^3)
\end{align*}
and a similar formula for $\tau$.
\end{proof}

In particular, Theorem~\ref{th:q.orthogonal} implies that $A_\rho(0)=A_{\tau}(0)=q^{*}$. 
Using basic numerical calculations, we get for $\kappa$ denoting  $\rho$ or $\tau$
\[
0.213< A_\kappa(r)< 0.271,
\]
for any $r\in (-1,1)$. Nevertheless, usually this bond is much tighter, which could already be observed in our previous numerical example (see e.g. Figure~\ref{fig:FRtype}). With some easy calculations, we get
\[
0.213< A_\kappa(r)< 0.230,
\]
for $r\in (-0.9,0.9)$. The graph of function $\Delta_\rho(q,r)=\rho_{[0,q]}(r) - \rho_{[q,1-q]}(r)$ for various fixed values of $q\in (0,0.5)$ is presented in Figure~\ref{fig:gr}, {see also Figure~\ref{fig:gr2} for the corresponding graph of  $\Delta_{\tau}$}.

\begin{remark}
When we consider the equilibrium state for conditional Spearman $\rho$ matrices (or Kendall $\tau$), we only need to know the dependance structure of $X$, given by it's copula. Thus, we can set any marginal distributions of $X_{1},\ldots, X_{n}$, without changing the equilibrium. This allow us to consider much more general class of multivariate distributions, for which the 20-60-20 rule will hold.
\end{remark}

\begin{figure}[!ht]
\begin{center}
\includegraphics[width=7cm]{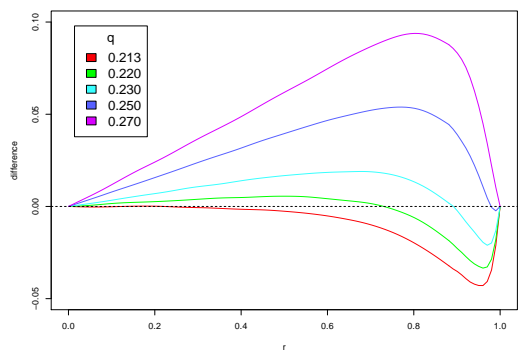}
\end{center}
\vspace{-20pt}
\caption{The graph of $\Delta_\rho(q,r)=\rho_{[0,q]}(r) - \rho_{[q,1-q]}(r)$ as function of $r$ for different values of (fixed) $q$.}\label{fig:gr}
\end{figure}

\begin{figure}[!ht]
\begin{center}
\includegraphics[width=7cm]{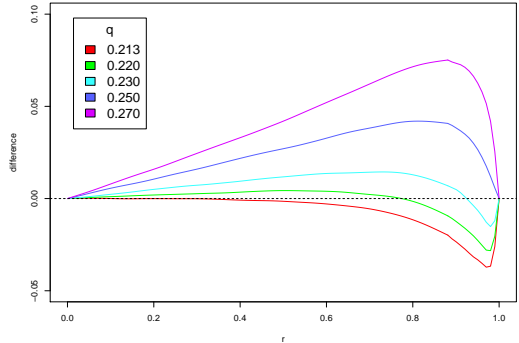}
\end{center}
\vspace{-20pt}
\caption{The graph of $\Delta_\tau(q,r)=\tau_{[0,q]}(r) - \tau_{[q,1-q]}(r)$ as function of $r$ for different values of (fixed) $q$.}\label{fig:gr2}
\end{figure}

\section{Abandoning Gaussian world}\label{S:non-normal}
When we loose the assumption that $X\sim \cN(\mu,\Sigma)$, the existence of equilibrium is no longer guaranteed. A natural question is if for any elliptical distribution the equivalent of 20/60/20 rule holds. In this section we will discuss this matter shortly.

We say that $X$ has the elliptic distribution if it can be defined in terms of a characteristic function
\begin{equation}\label{eq:elliptic1}
\phi_{X}(t)=e^{it'\mu}\Psi(t'\Sigma t),
\end{equation}
where $\mu$ is a vector (which coincides with mean vector, if it exists), $\Sigma$ is a scale matrix (which is proportional to covariance matrix, if it exists) and $\Psi$ is so called characteristic generator of the elliptical distribution (cf. ~\cite{GomGomMar2003} and references therein for a general survey about elliptic distributions). For simplicity, we will use so called stochastic representation of an elliptic distribution. It is well known (see \cite{GomGomMar2003}) that if $X$ has the density, 
then it is elliptic if and only if it can be presented as
\[
X=\mu +\sqrt{\Sigma}RU,
\]
where $\sqrt{\Sigma}$ is any square matrix such that $\sqrt{\Sigma}^{t}\sqrt{\Sigma}=\Sigma$ (e.g. obtained using Cholesky decomposition), $U$ is an $n$-dimensional random vector, uniformly distributed on the unit $n$-sphere, and $R$ is a nonnegative random vector, corresponding to the radial density, independent of $U$. Moreover, we will assume that the first two moments of $R$ exists, which ensures the existence of mean vector and covariance matrix of $X$. Now we can ask, if for given $U$ and $R$ the equilibrium state of $X$ always exists and if it is invariant wrt. $\mu$ and $\Sigma$.

Unfortunately, it is easy to show, that the equilibrium state (with covariance matrices) is not always achieved and the quasi-equilibrium state might strongly depend on $\Sigma$, even when we consider only the class of multivariate t-student distributions (i.e. we can consider appropriate radial distributions and covariance matrices in Algorithm~\ref{algorithm1}).

On the other hand, if we substitute covariance matrices with correlation matrices in \eqref{eq:balance}, then 
we are able to prove the results similar to Theorem~\ref{th:1} for a much more general class of elliptic distribution.

To illustrate this property, we have  conducted simple computational experiment, using multivariate t-student distribution, as it is commonly used by practitioners. Assuming $n=4$, for any $\nu\in\{2,3,\ldots,20\}$ we have picked 100 random matrices $\Sigma_{\nu}^{i}$ and for each $i=1,2,\ldots,100$ we simulated 1.000.000 Monte Carlo sample, assuming $X\sim t_{\nu}(0,\Sigma^{i}_{\nu})$. Next, we have calculated the values of $q^{i}_{\nu}\in (0,0.5)$, for which (quasi-)equilibrium state is attained (i.e. for estimates of conditional correlation matrices; see Algorithm~\ref{algorithm1}). In Figure~\ref{fig:q_wrt_df} we present the graph of $0.1$, $0.5$ and $0.9$ quantiles of the sample $\{q^{i}_{\nu}\}_{i=1}^{100}$ for $\nu=2,3,\ldots,20$.
The value of $q$ for which (quasi-)equilibrium state is achieved clearly depends on the degrees of freedom increasing to value $0.198$, which coincides with equilibrium state for multivariate normal distribution (i.e. note that t-student distribution converge to normal distribution, when $\nu\to\infty$).

\begin{figure}[!ht]
\begin{center}
\includegraphics[width=5cm]{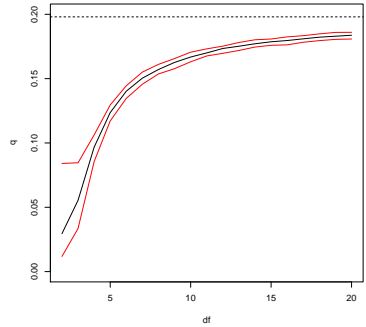}
\end{center}
\vspace{-20pt}
\caption{The graph of 0.1, 0.5 and 0.9 quantiles of $\{q^{i}_{\nu}\}_{i=1}^{100}$ for $\nu=2,3,\ldots,20$.
}\label{fig:q_wrt_df}
\end{figure}

\begin{algorithm}
\caption{Compute  quasi-equilibrium state for elliptic distribution}
\label{array-sum}
\begin{algorithmic}[1]
\Require
\Statex $n\in \bN_{+}$ -- dimension
\Statex $N\in\bN_{+}$ -- size of Monte Carlo sample
\Statex radial.Dist -- radial distribution (e.g. $\sqrt{\chi(n)}$ for multivariate normal)
\Procedure{Equilibrium}{$n$,$N$,radial.Dist}
    	\State Generate $U$: $N$ independent samples from $n$-dimensional unit sphere (uniform density)
      	\State Generate $R$: $N$ independent samples from (univariate) $radial.dist$
	\State Generate $\Sigma=\{\sigma_{ij}\}$: $n\times n$ scale matrix (proportional to covariance matrix)
	\While{\Big\{$\min_{i\neq j}\Big({\sigma^{2}_{ij}/|\sigma_{ii}\sigma_{jj}|}\Big)<0.2\Big\} \vee \Big\{\max_{i\neq j}\Big({\sigma^{2}_{ij}/|\sigma_{ii}\sigma_{jj}|}\Big)>0.8$\Big\} } 
	\State Generate (new) $\Sigma=\{\sigma_{ij}\}$: $n\times n$ scale matrix
	\EndWhile
	\State Compute $\sqrt{\Sigma}$, e.g. using Cholesky decomposition
    	\State Compute $X=\{X_{ik}\}=(\sqrt{\Sigma})^{'}RU$ (i.e. matrix $n\times N$; random sample from elliptic distribution)
	\State Define function $\textsc{Dist}(q)$, for $q\in (0,0.5)$
	\Function{Dist}{$q$}
    \State Compute $q^{1}$, sample lower $q$-quantile of $\{X_{1k}\}_{k=1}^{N}$
    \State Compute $q^{2}$ sample lower $(1-q)$-quantile of $\{X_{1k}\}_{k=1}^{N}$
    \State Compute conditional tail sample $X^{1}$, by selecting all $1\leq k\leq N$, for which $X_{1k}\leq q^{1}$
    \State Compute conditional central sample $X^{2}$, by selecting all $1\leq k\leq N$, for which $q^{1}\leq X_{1k}\leq q^{2}$
    \State Compute $\Sigma_{[0,q]}$, a (conditional) covariance matrix of $X^{1}$
    \State Compute $\Sigma_{[q,1-q]}$, a (conditional) covariance matrix of $X^{2}$
    \State Compute $d=\|\Sigma_{[0,q]}-\Sigma_{[q,1-q]}\|_{F}$
    \State
    \Return $d$
\EndFunction
 \State Compute $\hat{q}=\argmin_{0\leq q\leq 0.5} \textsc{Dist}(q)$
 \State
 \Return $\hat{q}$
\EndProcedure
\end{algorithmic}
\label{algorithm1}
\end{algorithm}

\pagebreak

\section*{Acknowledgments}
Marcin Pitera acknowledges the support by Project operated within the Foundation for Polish Science IPP Programme "Geometry and Topology in Physical Models" co-financed by the EU European Regional Development Fund, Operational Program Innovative Economy 2007-2013.

{\small
\bibliographystyle{amsplain}

\begin{thebibliography}{10}

\bibitem{Ann2001}
Susan Annunzio, \emph{e-{L}eadership -- {P}roven {T}echniques for {C}reating an
  {E}nvironment of {S}peed and {F}lexibility in the {D}igital {E}conomy}, Human
  Resource Management \textbf{40} (2001), no.~4, 381--383.

\bibitem{Bach2005}
David Bach, \emph{Start late, finish rich: A no-fail plan for achieving
  financial freedom at any age}, Crown Business, 2005.

\bibitem{Bol2004}
B.~Bolger, \emph{Ten steps to designing an effective incentive program},
  Employment Relations Today \textbf{31} (2004), no.~1, 25--33.

\bibitem{Cre1993}
James Creelman, \emph{An act of faith}, The TQM Magazine \textbf{5} (1993),
  no.~3.

\bibitem{Dup2002}
David Dupper, \emph{School social work: Skills and interventions for effective
  practice}, John Wiley \& Sons, 2002.

\bibitem{GomGomMar2003}
E.~G{\'o}mez, M.~A G{\'o}mez-Villegas, and J.~M. Mar{\'\i}n, \emph{A survey on
  continuous elliptical vector distributions}, Revista matem{\'a}tica
  complutense \textbf{16} (2003), no.~1, 345--361.

\bibitem{GreWod2005}
M.~A. Grey and A.~C. Woodrick, \emph{Latinos have revitalized our community:
  {M}exican migration and {A}nglo responses in {M}arshalltown, {I}owa}, New
  destinations: Mexican immigration in the United States (2005), 133--154.

\bibitem{Hin2002}
L.~M. Hinman, \emph{The impact of the internet on our moral lives in academia},
  Ethics and information technology \textbf{4} (2002), no.~1, 31--35.

\bibitem{JawPit2014}
P.~Jaworski and M.~Pitera, \emph{On spatial contagion and multivariate {GARCH}
  models}, Applied Stochastic Models in Business and Industry \textbf{30}
  (2014), no.~3, 303--327.

\bibitem{JohKotBal1994}
N.~L. Johnson, S.~Kotz, and N.~Balakrishnan, \emph{Continuous univariate
  distributions}, vol.~1, 1994.

\bibitem{KamMakKob2005}
A.~Kamiya, F.~Makino, and S.~Kobayashi, \emph{Worker ants' rule-based genetic
  algorithms dealing with changing environments}, Soft Computing in Industrial
  Applications, 2005. SMCia/05. Proceedings of the 2005 IEEE Mid-Summer
  Workshop on, IEEE, 2005, pp.~117--121.

\bibitem{Nel2007}
R.~B. Nelsen, \emph{An introduction to copulas}, Springer, 2006.

\bibitem{Rob2009}
Les Robinson, \emph{A summary of diffusion of innovations}, Enabling Change
  (2009).

\bibitem{SlaHol2007}
J.~Slagell and J.~Holtermann, \emph{Climbing peaks and navigating valleys:
  Managing personnel from high altitude}, The Serials Librarian \textbf{52}
  (2007), no.~3-4, 271--275.

\bibitem{Tyn1999}
Susan~A Tynan, \emph{Best behaviors}, Management Review \textbf{88} (1999),
  no.~10, 58--61.

\end{thebibliography}
\providecommand{\bysame}{\leavevmode\hbox to3em{\hrulefill}\thinspace}
\providecommand{\MR}{\relax\ifhmode\unskip\space\fi MR }
\providecommand{\MRhref}[2]{%
  \href{http://www.ams.org/mathscinet-getitem?mr=#1}{#2}
}
\providecommand{\href}[2]{#2}

}

 \end{document}